\documentclass[12pt]{amsart}
\usepackage{a4wide,enumerate,xcolor,graphicx}
\usepackage{amsmath,comment}
\usepackage{scrextend} 
\allowdisplaybreaks

\let\pa\partial
\let\na\nabla
\let\eps\varepsilon
\newcommand{\N}{{\mathbb N}}
\newcommand{\R}{{\mathbb R}}
\newcommand{\diver}{\operatorname{div}}
\newcommand{\T}{{\mathcal T}}
\newcommand{\E}{{\mathcal E}}
\newcommand{\m}{\operatorname{m}}
\newcommand{\Eint}{{\mathcal E}_{{\rm int},K}}
\newcommand{\dist}{{\operatorname{d}}}
\newcommand{\M}{{\mathcal M}}
\newcommand{\F}{{\mathcal F}}
\newcommand{\DD}{{\mathrm D}}

\newtheorem{theorem}{Theorem}
\newtheorem{lemma}[theorem]{Lemma}

\newtheorem{remark}[theorem]{Remark}

\begin{document}

\title[Analysis of a finite-volume scheme]{Analysis of a finite-volume scheme \\
for a single-species biofilm model}

\author[C. Helmer]{Christoph Helmer}
\address{Institute for Analysis and Scientific Computing, Vienna University of
	Technology, Wiedner Hauptstra\ss e 8--10, 1040 Wien, Austria}
\email{christoph.helmer@tuwien.ac.at}

\author[A. J\"ungel]{Ansgar J\"ungel}
\address{Institute for Analysis and Scientific Computing, Vienna University of
	Technology, Wiedner Hauptstra\ss e 8--10, 1040 Wien, Austria}
\email{juengel@tuwien.ac.at}

\author[A. Zurek]{Antoine Zurek}
\address{Laboratoire de Math\'ematiques Appliqu\'ees de Compi\`egne, EA2222, Sorbonne Universit\'e--Universit\'e de Technologie de Compi\`egne (UTC), Compi\`egne, France}
\email{antoine.zurek@utc.fr}

\date{\today}

\thanks{The authors have been partially supported by the Austrian Science Fund (FWF), 
grants P30000, P33010, F65, and W1245, and by the multilateral project of the
Austrian Agency for International Cooperation in Education and Research
(OeAD), grants FR 01/2021 and MULT 11/2020. 
This work received funding from the European 
Research Council (ERC) under the European Union's Horizon 2020 research and 
innovation programme, ERC Advanced Grant NEUROMORPH, no.~101018153.}

\begin{abstract}
An implicit Euler finite-volume scheme for a parabolic reaction-diffusion system 
modeling biofilm growth is analyzed and implemented. The system consists of a 
degenerate-singular diffusion equation for the biomass fraction, which is coupled 
to a diffusion equation for the nutrient concentration, and it is solved in a bounded 
domain with Dirichlet boundary conditions. By transforming the biomass fraction to an
entropy-type variable, it is shown that 
the numerical scheme preserves the lower and upper bounds of the biomass fraction. 
The existence and uniqueness of a discrete solution and the convergence of the scheme 
are proved. Numerical experiments in one and two space dimensions illustrate,
respectively, the rate of convergence in space of our scheme and the 
temporal evolution of the biomass fraction and the nutrient concentration.
\end{abstract}



\maketitle


\section{Introduction}

Biofilms are accumulations of microorganisms that grow on surfaces in liquids
and can be prevalent in natural, industrial, and hospital environments 
\cite{WEMNPRL06}. They can form, for instance, 
on teeth as dental plaque and on inert surfaces of implanted
devices like catheters. Another example are biofilms grown on filters, 
which may extract and digest
organic compounds and help to clean wastewater. A biofilm growth model
that well describes the spatial spreading mechanism for biomass and the 
dependency on the nutrient was suggested in \cite{EPL01}. The model
was analyzed in \cite{EZE09,GKN18} and numerically solved in \cite{AES18,DuEb06}.
Up to our knowledge, there does not exist any analysis for the
numerical approximations in the literature. In this paper, we provide such an
analysis for an implicit Euler finite-volume scheme for the model in \cite{EZE09}.

The biofilm is modeled by the biomass fraction $M(x,t)$ and the nutrient 
concentration $S(x,t)$, satisfying the diffusion equations
\begin{align}\label{1.eqS}
  \pa_t S - d_1\Delta S &= g(S,M), \\ 
	\pa_t M - d_2\diver(f(M)\na M) &= h(S,M)\quad\mbox{in }\Omega,\ t>0, \label{1.eqM}
\end{align}
and the initial and boundary conditions
\begin{equation}\label{1.bic}
  S(0) = S^0,\ M(0) = M^0\quad\mbox{in }\Omega, \quad
	S=1,\ M=M^D\quad\mbox{on }\pa\Omega,\ t>0,
\end{equation}
where $\Omega\subset\R^d$ ($d\ge 1$) is a bounded domain and $0<M^D<1$. 
Other boundary conditions can also be considered; see Remark \ref{rem.hypo}.

The nutrients are consumed with the Monod reaction rate $g(S,M)$, while
biomass is produced by the production rate $h(S,M)$ that is the sum of a Monod 
reaction term and a wastage term,
\begin{equation}\label{1.source}
  g(M,S) = -\kappa_1\frac{SM}{\kappa_4+S}, \quad
	h(M,S) = \kappa_3\frac{SM}{\kappa_4+S} - \kappa_2 M,
\end{equation}
where $\kappa_i\ge 0$ for $i=1,2,3$ and $\kappa_4>0$. The diffusion coefficients
$d_1$ and $d_2$ are assumed to be positive numbers. 
Postulating that there is a sharp biomass front, spatial spreading occurs only
when there is a significant amount of biomass, and the biomass fraction cannot
exceed the maximum bound $M_{\rm max}=1$, the authors of \cite{EPL01} have 
suggested the density-dependent diffusion term
\begin{equation}\label{1.f}
  f(M) = \frac{M^b}{(1-M)^a}, \quad\mbox{where }a>1,\ b>0.
\end{equation}
The diffusion operator in \eqref{1.eqM} can be written as
$\diver(f(M)\na M)=\Delta F(M)$, where
\begin{equation}\label{1.F}
  F(M) = \int_0^M f(s)ds, \quad M\ge 0,
\end{equation}
which gives a porous-medium degeneracy for $M$ close to zero.
This degeneracy leads to a finite speed of propagation and is 
responsible for the formation of a sharp interface between the biofilm and 
the surrounding liquid. The superdiffusion singularity forces the biomass 
fraction to be smaller than the maximal amount $M_{\rm max}=1$. 

The aim of this paper is to analyze an implicit Euler finite-volume scheme
for \eqref{1.eqS}--\eqref{1.f} that preserves the bounds 
$0\le S\le 1$ and $0\le M<1$. We show the existence of a discrete solution,
prove the convergence of the scheme, and present some numerical tests
in one and two space dimensions. The main difficulty of the analysis is the
degenerate-singular diffusion term. On the continuous level, 
if $M^0\le 1-\eps_0$ in $\Omega$ and $M^D\le 1-\eps_0$ for some $\eps_0\in(0,1)$ then
the comparison principle implies that there exists $\delta(\eps_0)>0$ such that
$M\le 1-\delta(\eps_0)$ in $\Omega$ \cite[Prop.~6]{EZE09}. Unfortunately,
we have not found any suitable comparison principle on the discrete level. 

We overcome this issue by using two ideas. First, we formulate equation \eqref{1.eqM}
for the biomass in terms of the approximate ``entropy variable'' \cite{Jue16}
$$
  W^\eps_K := F(M^\eps_K) - F(M^D) + \eps\log\frac{M^\eps_K}{M^D},
$$
where $K\subset\Omega$ denotes a control volume and $\eps>0$ is a regularization
parameter. 
For given $W^\eps_K\in\R$, the biomass fraction is defined implicitly by 
the invertible mapping $(0,1)\to\R$, $M^\eps\mapsto W^\eps$. 
The advantage is that the bounds $0<M^\eps_K<1$ are guaranteed 
by this definition. In fact, the singularity in $F$ provides the upper bound,
while the $\eps$-regularization gives the lower bound.
Second, we prove an $\eps$-uniform bound for $F(M^\eps)$ in $L^1(\Omega)$, which
shows that the a.e.\ limit $M_K=\lim_{\eps\to 0}M^\eps_K$ satisfies $0\le M_K<1$
for all control volumes $K$.

The original biofilm model of \cite{EPL01} contains the transport term
$u\cdot\na M$ in the equation for the biofilm fraction. The flow velocity $u$
is assumed to satisfy the incompressible Navier--Stokes equations in the
region $\{M=0\}$, while $u=0$ in $\{M>0\}$. Thus, model \eqref{1.eqS}--\eqref{1.eqM}
implicitly assumes that $M>0$. We do not require this condition but we prove
in Theorem \ref{thm.ex} below that this property is fulfilled if $M^0$
and $M^D$ are strictly positive. 

The existence and uniqueness of a global weak solution to
\eqref{1.eqS}--\eqref{1.f} was shown in \cite{EZE09}, while the original model
was analyzed in \cite{GKN18} by formulating it as a system of variational
inequalities. The model of \cite{EPL01} was extended in \cite{EEWZ14} by
taking into account nutrient taxis, which forces the biofilm to move up a
nutrient concentration gradient. In that work, a fast-diffusion exponent $a\in(0,1)$
instead of a superdiffusive value $a>1$ (like in \cite{EPL01}) was considered.
Equations \eqref{1.eqS}--\eqref{1.f} were numerically solved using finite differences
\cite{EPL01} or finite volumes \cite{AES18} but without any analysis. 
Some properties of the semi-implicit 
Euler finite-difference scheme were shown in \cite{EbDe07}. A finite-element
approximation for \eqref{1.eqM} with linear diffusion $f(M)=1$ but a constraint
on the upper bound for the biomass was suggested in \cite{AlPe20}.

Local mixing effects between different biofilm species can be described by
multispecies biofilm models \cite{RSE15}. The resulting
cross-diffusion system for the biofilm proportions (without nutrient equation)
was analyzed in \cite{DMZ19} and numerically investigated in \cite{DJZ21}.
A nutrient equation was included in a two-species biofilm system in \cite{GSE18},
where a time-adaptive scheme was suggested to deal with biomasses close to the
maximal value. A finite-volume method was proposed in \cite{RaEb14} for a
biofilm system for the active and inert biomasses, completed by equations
for the nutrient and biocide concentrations, but without performing a 
numerical analysis. 

Let us mention also related biofilm models. The first model was suggested by
Wanner and Gujer \cite{WaGu86} and consists of a one-dimensional
transport equation for the
biofilm species together with a differential equation for the biofilm thickness.
A nonlinear hyperbolic system for the formation of biofilms was derived in
\cite{CRNR13}. Other works were concerned with diffusion equations coupled to
a fluiddynamical model as in \cite{EPL01}. 
For instance, the paper \cite{Sch19} provides a formal 
derivation of the diffusion equations for the biomass and
nutrient, coupled to the Darcy--Stokes equation for the fluid velocity. 
Numerical simulations of a gradient-flow system
for the dead and live biofilm bacteria, coupled to the incompressible Navier--Stokes
equations for the fluid velocity, were presented in \cite{Zha12}, based on a
Crank--Nicolson discretization and an upwinding scheme.

With the exception of \cite{AlPe20,DJZ21}, 
these mentioned works do not contain any analysis of the numerical scheme. 
The paper \cite{AlPe20} is concerned with a finite-element method
and assumes linear diffusion, while \cite{DJZ21} does not contain an equation
for the nutrient. In this paper, we provide a numerical analysis of a finite-volume
scheme to \eqref{1.eqS}--\eqref{1.eqM} for the first time.
Our results can be sketched as follows:
\begin{itemize}
\item We prove the existence of a finite-volume solution $(S_K^k,M_K^k)$,
where $K$ denotes a control volume and $k$ is the time step, satisfying the
bounds $0\le S_M^k\le 1$ and $0\le M_K^k<1$ for all control volumes $K$ and
all time steps $k$.
\item If the initial and boundary biomass are strictly positive, we obtain the
uniqueness of a discrete solution.
\item The discrete solution converges in a certain sense, for mesh sizes
$(\Delta x,\Delta t)\to 0$, to a weak solution to \eqref{1.eqS}--\eqref{1.F}.
\end{itemize} 

The paper is organized as follows.
The numerical scheme and the main results are formulated in Section \ref{sec.main}.
Section \ref{sec.ex} is concerned with the existence proof (Theorem \ref{thm.ex}),
while the uniqueness result (Theorem \ref{thm.unique}) is shown in Section
\ref{sec.unique}. The convergence of the scheme requires uniform estimates
which are proved in Section \ref{sec.est}. The convergence result
(Theorem \ref{thm.conv}) is then shown in Section \ref{sec.conv}. Numerical simulations are presented in Section \ref{sec.num}.


\section{Numerical scheme and main results}\label{sec.main}

\subsection{Notation and assumptions}

Let $\Omega\subset\R^2$ be an open,  bounded, polygonal domain.
We consider only two-dimensional domains, but the generalization
to higher space dimensions is straightforward.
An admissible mesh of $\Omega$ is given by a family $\T$
of open polygonal control volumes (or cells), a family $\E$ of edges, and
a family ${\mathcal P}$ of points $(x_K)_{K\in\T}$ associated to the
control volumes and satisfying Definition 9.1 in \cite{EGH00}. This definition
implies that the straight line $\overline{x_Kx_L}$ between two centers of
neighboring cells is orthogonal to the edge $\sigma=K|L$ between two cells.
The condition is satisfied, for instance, by triangular meshes whose triangles
have angles smaller than $\pi/2$ \cite[Example 9.1]{EGH00} or by Vorono\"{\i} meshes
\cite[Example 9.2]{EGH00}.

The family of edges $\E$ is assumed to consist of interior edges $\E_{\rm int}$
satisfying $\sigma\subset\Omega$ and boundary edges $\sigma\in\E_{\rm ext}$ fulfilling
$\sigma\subset\pa\Omega$. For a given control volume $K\in\T$, 
we denote by $\E_K$ the set of edges of $K$. This set splits into 
$\E_K=\Eint\cup\E_{{\rm ext},K}$. 
For any $\sigma\in\E$, there exists at least one cell $K\in\T$ such that 
$\sigma\in\E_K$. 
When $\sigma$ is an interior
cell, $\sigma=K|L$, $K_\sigma$ can be either $K$ or $L$.

The admissibility of the mesh and the fact that $\Omega$ is two-dimensional 
imply that
\begin{equation}\label{2.estmesh}
  \sum_{K\in\T}\sum_{\sigma\in\E_K}\m(\sigma)\dist(x_K,\sigma)
	\le 2\sum_{K\in\T}\m(K) = 2\m(\Omega),
\end{equation}
where d is the Euclidean distance in $\R^2$ and $\m$ is the one- or two-dimensional 
Lebesgue measure. Let $\sigma\in\E$ be an edge. We define the distance
$$
  \dist_\sigma = \left\{\begin{array}{ll}
	\dist(x_K,x_L) &\quad\mbox{if }\sigma=K|L\in\E_{{\rm int},K}, \\
	\dist(x_K,\sigma) &\quad\mbox{if }\sigma\in\E_{{\rm ext},K},
	\end{array}\right.
$$ 
and introduce the transmissibility coefficient by
\begin{equation}\label{2.trans}
  \tau_\sigma = \frac{\m(\sigma)}{\dist_\sigma}.
\end{equation}
We assume that the mesh satisfies the following regularity assumption: There exists
$\xi>0$ such that for all $K\in\T$ and $\sigma\in\E_K$,
\begin{equation}\label{2.regmesh}
  \dist(x_K,\sigma)\ge \xi \dist_\sigma.
\end{equation}
The size of the mesh is denoted by $\Delta x=\max_{K\in\T}\operatorname{diam}(K)$.

Let $T>0$ be the end time, $N_T\in\N$ the number of time steps, $\Delta t=T/N_T$ 
the time step size, and set $t_k=k\Delta t$ for $k=0,\ldots,N_T$. 
We denote by ${\mathcal D}$ an admissible space-time discretization of 
$\Omega_T:=\Omega\times(0,T)$, 
composed of an admissible mesh ${\mathcal T}$ and the values $(\Delta t,N_T)$. 
The size of ${\mathcal D}$ is defined by $\eta:=\max\{\Delta x,\Delta t\}$.

As it is usual for the finite-volume method, we introduce functions that are
piecewise constant in space and time. The finite-volume scheme yields a vector
$v_\T=(v_K)_{K\in\T}\in\R^{\#\T}$ of approximate values of a piecewise
constant function $v$ such that $v = \sum_{K\in\T}v_K\mathbf{1}_K$,
where $\mathbf{1}_K$ is the characteristic function of $K$. 
We write $v_\M=(v_\T,v_\E)$ for the vector that contains the approximate values
in the control volumes and on the boundary edges, where 
$v_\E:=(v_\sigma)_{\sigma\in\E_{{\rm ext}}}\in\R^{\#\E_{\rm ext}}$.
For such a vector, we use the notation
\begin{equation}\label{2.vKsigma}
  v_{K,\sigma} = \left\{\begin{array}{ll}
	v_L &\quad\mbox{if }\sigma=K|L\in\E_{{\rm int},K}, \\
	v_\sigma &\quad\mbox{if }\sigma\in\E_{{\rm ext},K}
	\end{array}\right.
\end{equation}
for $K\in\T$ and $\sigma\in\E_K$ and introduce the discrete gradient
\begin{equation}\label{2.Dsigma}
  \DD_\sigma  v := |\DD_{K,\sigma} v|, \quad\mbox{where }
	\DD_{K,\sigma} v = v_{K,\sigma}-v_K.
\end{equation}

The discrete $H^1(\Omega)$ seminorm and the discrete $H^1(\Omega)$
norm are defined by
\begin{equation}\label{2.norm}
  |v|_{1,2,\M} = \bigg(\sum_{\sigma\in\E}\tau_\sigma(\DD_\sigma v)^2\bigg)^{1/2},
	\quad \|v\|_{1,2,\M} = \big(\|v\|^2_{0,2,\M} + |v|^2_{1,2,\M}\big)^{1/2},
\end{equation}
where $\|\cdot\|_{0,p,\M}$ denotes the $L^p(\Omega)$ norm
$$
  \|v\|_{0,p,\M} = \bigg(\sum_{K\in\T}\m(K)|v_K|^p\bigg)^{1/p} 
	\quad \mbox{for }1 \le p < \infty.
$$
Then, for a given family of vectors $v^k=(v^k_\T,v^k_\E)$ for $k=1,\ldots,N_T$ and 
a given nonnegative constant $v^D$ such that $v^k_\sigma = v^D$ for all 
$\sigma \in \E_{\rm ext}$, we define the piecewise 
constant in space and time function $v$ by
\begin{equation}\label{reconstruc}
  v(x,t) = \sum_{K\in\T}v_K^k\mathbf{1}_K(x)
	\quad\mbox{for }x\in\Omega,\ t\in(t_{k-1},t_k], \, k=1,\ldots,N_T.
\end{equation}

For the definition of an approximate gradient for such functions, we need to
introduce a dual mesh. Let $K\in\T$ and $\sigma\in\E_K$. The cell $T_{K,\sigma}$
of the dual mesh is defined as follows:
\begin{itemize}
\item If $\sigma=K|L\in\E_{{\rm int},K}$, then $T_{K,\sigma}$ is that cell 
(``diamond'') whose vertices are given by $x_K$, $x_L$, and the end points of the 
edge $\sigma$.
\item If $\sigma\in\E_{{\rm ext},K}$, then $T_{K,\sigma}$ is that cell 
(``half-diamond'') whose vertices are given by $x_K$ and the end points of the 
edge $\sigma$.
\end{itemize}
An example of a construction of a dual mesh can be found in \cite{CLP03}.
The cells $T_{K,\sigma}$ define, up to a negligible set, a partition of $\Omega$. 
The definition of the dual mesh implies the following property. As the straight line 
between two neighboring centers of cells $\overline{x_Kx_L}$ is orthogonal to the edge 
$\sigma=K|L$, it follows that
\begin{equation}\label{2.para}
  \m(\sigma)\dist(x_K,x_L) = 2\m(T_{K,\sigma}) \quad\mbox{for all }
	\sigma=K|L\in \E_{{\rm int},K}.
\end{equation}
The approximate gradient of a piecewise constant function $v$ in $\Omega_T$ is 
given by
$$
  \na^{\mathcal D} v(x,t) = \frac{\m(\sigma)}{\m(T_{K,\sigma})}\DD_{K,\sigma} 
	v^k \nu_{K,\sigma}
	\quad\mbox{for }x\in T_{K,\sigma},\ t\in(t_{k-1},t_k], , k=1,\ldots,N_T,
$$
where the discrete operator $\DD_{K,\sigma}$ is given in \eqref{2.Dsigma} 
and $\nu_{K,\sigma}$ is the unit vector that is normal to $\sigma$ and that
points outward of $K$.


\subsection{Numerical scheme}

We are now in the position to formulate the finite-volume discretization of
\eqref{1.eqS}--\eqref{1.bic}. Let ${\mathcal D}$ be an admissible discretization 
of $\Omega_T$. The initial conditions are discretized by the averages
\begin{equation}\label{sch.ic}
  S_K^0 = \frac{1}{\m(K)}\int_K S^0(x)dx, \quad 
	M_K^0 = \frac{1}{\m(K)}\int_K M^0(x)dx \quad\mbox{for }K\in\T.
\end{equation}
On the Dirichlet boundary, we set $S_\sigma^k=1$ and $M_\sigma^k=M^D$ for 
$\sigma \in \E_{\rm ext}$ at time $t_k$.

Let $S_K^k$ and $M_K^k$ be some approximations of the mean values of $S(\cdot,t_k)$ 
and $M(\cdot,t_k)$, respectively, in the cell $K$. Then the elements 
$S_K^k$ and $M^k_K$ are solutions to
\begin{align}
	& \frac{\m(K)}{\Delta t}(S_K^{k}-S_K^{k-1})
	+ \sum_{\sigma\in\E_K}\F_{S,K,\sigma}^k 
	= \m(K) g(S_K^k,M_K^k), \label{schS} \\
	& \frac{\m(K)}{\Delta t}(M_K^{k}-M_K^{k-1})
	+ \sum_{\sigma\in\E_K}\F_{M,K,\sigma}^k 
	= \m(K)h(S_K^k,M_K^k), \label{schM}
\end{align}
the numerical fluxes are defined as
\begin{equation}
   \F_{S,K,\sigma}^k = -\tau_\sigma d_1 \DD_{K,\sigma} S^k, \quad
   \F_{M,K,\sigma}^k = -\tau_\sigma d_2 \DD_{K,\sigma}F(M^k), \label{Flux}
\end{equation}
where $K\in\T$, $\sigma\in\E_K$, $k\in\{1,\ldots,N_T\}$, and we recall definitions 
\eqref{1.source} for $g$ and $h$, 
\eqref{1.F} for $F$, and \eqref{2.trans} for $\tau_\sigma$. 

For the convenience of the reader, we recall the discrete integration-by-parts
formula for piecewise constant functions $v=(v_\T,v_\E):$
\begin{equation}\label{2.ibp}
  \sum_{K\in\T}\sum_{\sigma\in\E_K}\F_{K,\sigma} v_K 
	= -\sum_{\sigma\in\E}\F_{K,\sigma}\DD_{K,\sigma}v
	+ \sum_{\sigma\in\E_{\rm ext}}\F_{K,\sigma} v_{\sigma},
\end{equation}
where $\F_{K,\sigma}$ is a numerical flux like in \eqref{Flux}.


\subsection{Main results}

We impose the following hypotheses:

\begin{labeling}{(H3)}
\item[(H1)] Domain: $\Omega\subset\R^2$ is a bounded polygonal domain.

\item[(H2)] Discretization: ${\mathcal D}$ is an admissible discretization of 
$\Omega_T:=\Omega\times(0,T)$ satisfying the regularity condition \eqref{2.regmesh}.

\item[(H3)] Initial data: $S^0$, $M^0 \in L^2(\Omega)$ satisfy
$0\le S^0 \le 1$ and $0\le M^0 < 1$ in $\Omega$. 

\item[(H4)] Dirichlet datum: $0 < M^D < 1$.

\item[(H5)] Parameters: $d_1$, $d_2 > 0$, $\kappa_i\ge 0$ for 
$i= 1,2,3$, $\kappa_4 > 0$, $a \ge 1$, and $b \ge 0$.
\end{labeling}

\begin{remark}[Discussion of the hypotheses]\label{rem.hypo}\rm
Conditions $M^0<1$ and $M^D<1$ allow for the proof of $M_K^k<1$ for all $K\in\T$
and $k=1,\ldots,N_T$, thus avoiding quenching of the solution, i.e.\ the
occurrence of regions with $M_K^k=1$. We assume that $M^D$ is positive to be
able to introduce an entropy variable. This condition can be relaxed by
introducing an approximation procedure. The assumption that the boundary biomass
is constant is imposed for simplicity. It can be generalized to piecewise
constant or time-dependent boundary data, for instance. 
Moreover, mixed Dirichlet--Neumann boundary conditions
for the biomass could be imposed as well; see \cite[Section 4]{EZE09}. 
On the other hand, pure Neumann boundary conditions for $M$ may lead, 
in the continuous case, to a quenching phenomenon in finite time, as shown in
\cite{EZE09}. We may assume that the diffusion coefficents $d_1$ and $d_2$ depend
on the spatial variable if $d_1(x)$ and $d_2(x)$ are strictly positive.
The condition $a\ge 1$ corresponds to ``very fast diffusion''.
In numerical simulations, usually the values $a=b=4$ are chosen \cite[Table 1]{EPL01}.
\qed\end{remark}

Our first main result concerns the existence of solutions to the numerical scheme.
We introduce the function
\begin{equation}\label{2.defZ}
  Z(M) := \int_{M^D}^M F(s)ds - F(M^D)(M_K-M^D), \quad M\in[0,1).
\end{equation}

\begin{theorem}[Existence of discrete solutions]\label{thm.ex}
Assume that Hypotheses (H1)--(H5) hold.
Then, for every $k=1,\ldots,N_T$, there exists a solution $(S^k, M^k)$ 
to scheme \eqref{sch.ic}--\eqref{Flux} satisfying
\begin{equation}\label{2.bounds.thm}
  0 \le S^k_K \le 1, \quad 0 \le  M^k_K < 1 \quad \mbox{for all }K \in \T,
\end{equation}
and there exist positive constants $C_1$ and $C_2$ independent of $\Delta x$ 
and $\Delta t$ such that
\begin{equation}\label{2.estimFM}
  \|Z(M^k)\|_{0,1,\M}
	+ \Delta t C_1\|F(M^k)\|_{1,2,\M}^2
	\le \|Z(M^{k-1})\|_{0,1,\M}
	+ \Delta t C_2.
\end{equation}
Moreover, if $M^0 \ge m_0$ in $\Omega$ and $M_D \ge m_0$ for some $m_0>0$ then
\begin{equation}\label{2.strict.pos}
  M^k_K \ge m_0 \exp(-\kappa_2 t_k) \quad \mbox{for all }K \in \T, \ k=1,\ldots, N_T.
\end{equation}
\end{theorem}

The existence result is proved by a fixed-point argument based on a topological
degree result. The main difficulty is to approximate the equations in such
a way that the singular point $M=1$ is avoided. This can be done,
as in \cite{EZE09}, by introducing a cut-off approximation $f_\eps(M)$ of $f(M)$.
Then, by the comparison principle, it is possible to show the bound
$M^\eps\le 1-\delta(\eps)$ for the approximate biomass $M^\eps$, where $\delta(\eps)
\in(0,1)$.
Since the comparison principle cannot be easily extended to the discrete case, 
we have chosen another approach. We introduce the ``entropy variable''
$W^\eps_K:=Z_\eps'(M^\eps_K)$, where $Z_\eps$ is the sum of $Z(M^\eps_K)$ and
$\eps$ times the Boltzmann entropy (see \eqref{3.defZ}). Then 
$0<M^\eps_K<1$ by definition and we can derive a uniform estimate similar to
\eqref{2.estimFM}. The uniform bound for $F(M^\eps)$ 
allows us to infer that the a.e.\ limit function
$M_K=\lim_{\eps\to 0}M^\eps_K$ satisfies $M_K<1$ for all $K\in\T$.
The positive lower bound for $M^k$ comes from the fact that the source term 
$h(S_K^k,M_K^k)$ is bounded from below by the linear term $-\kappa_2M_K^k$,
and it is proved by a Stampacchia truncation method.

\begin{theorem}[Uniqueness of discrete solutions]\label{thm.unique}
Assume that Hypotheses (H1)--(H5) hold and that there exists a constant 
$m_0>0$ such that $M^0(x) \ge m_0$ for $x \in \Omega$ and $M^D \ge m_0$.
Then there exists $\gamma^*>0$, depending on the data, the mesh, and $m_0$, 
such that for all $0<\Delta t<\gamma^*$, 
there exists a unique solution to scheme \eqref{sch.ic}--\eqref{Flux}.
\end{theorem}

The proof of the theorem is based on a discrete version of the dual method.
On the continuous level, the idea is to choose test functions $\psi$ and $\phi$
solving $-\Delta\psi=S_1-S_2$ and $-\Delta\phi=M_1-M_2$ with homogeneous
Dirichlet boundary data, where $(S_1,M_1)$
and $(S_2,M_2)$ are two solutions to \eqref{1.eqS}--\eqref{1.eqM}
with the same initial data, and to exploit
the monotonicity of the nonlinearity $F(M)$. On the discrete level,
we replace the diffusion equations for $\psi$ and $\phi$ by the corresponding
finite-volume schemes and estimate similarly as in the continuous case.
The restriction on the time step size is due to $L^2(\Omega)$ estimates
coming from the source terms.

We also prove that our scheme converges to the continuous model,
up to a subsequence. For this result,
we introduce a family $({\mathcal D}_m)_{m\in\N}$ of admissible space-time 
discretizations of $\Omega_T$ indexed by the size 
$\eta_m=\max\{\Delta x_m,\Delta t_m\}$ of the mesh, satisfying $\eta_m\to 0$
as $m\to\infty$. We denote by $\M_m$ the corresponding meshes of $\Omega$ and
by $\Delta t_m$ the corresponding time step sizes.  
Finally, we set $\na^m:=\na^{\mathcal{D}_m}$.

\begin{theorem}[Convergence of the scheme]\label{thm.conv}
Assume that the Hypotheses (H1)--(H5) hold. 
Let $({\mathcal D}_m)_{m \in \N}$ be a family of admissible meshes satisfying 
\eqref{2.regmesh} 
uniformly and let 
$(S_m,M_m)_{m\in\N}$ be a corresponding sequence of finite-volume solutions to 
scheme \eqref{sch.ic}--\eqref{Flux} constructed in Theorem \ref{thm.ex}. 
Then there exist $(S,M)\in L^\infty(\Omega_T;\R^2)$ 
and a subsequence of $(S_m,M_m)$ (not relabeled)
such that, as $m \to \infty$, 
\begin{align*}
	S_m \rightarrow S, \quad M_m\to M \quad &\text{a.e. in } \Omega_T, \\
	\na^{m} S_m \rightharpoonup \na S, \quad 
	\na^{m} F(M_m) \rightharpoonup \na F(M) \quad &\text{weakly in } L^2(\Omega_T).
\end{align*}
The functions $S-1$ and $F(M)-F(M^D)$ belong to the space $L^2(0,T;H_0^1(\Omega))$. Moreover, the limit $(S,M)$ is a weak solution to \eqref{1.eqS}--\eqref{1.bic}, i.e., 
for all $\psi$, $\phi \in C_0^\infty(\Omega \times [0,T))$, 
\begin{align}\label{weakS}
	-\int_0^T &\int_\Omega S \pa_t \psi dx dt - 
	\int_\Omega S^0(x) \psi(x,0) dx 
	+ d_1\int_0^T \int_\Omega \na S \cdot \na \psi dx dt \\
	&= \int_0^T \int_\Omega g(S,M) \psi dx dt, \nonumber \\
  \label{weakM}
	-\int_0^T & \int_\Omega M \pa_t \phi dx dt
	- \int_\Omega M^0(x) \phi(x,0) dx 
	+ d_2\int_0^T \int_\Omega \na F(M) \cdot \na \phi dx dt \\ 
	&= \int_0^T\int_\Omega h(S,M) \phi dx dt. \nonumber
\end{align}
\end{theorem}

The convergence proof is based on the uniform estimates derived for the proof of Theorem \ref{thm.ex} and a discrete compensated compactness technique 
\cite{ACM17} needed to identify the nonlinear limits. 
For the limit $m\to\infty$, we use the techniques of \cite{CLP03}. 
If uniqueness for the limiting model holds in the class of weak solutions,
the whole sequence $(S_m,M_m)$ converges. Uniqueness in a smaller class of
functions is proved \cite[Theorem 3.2]{EZE09}, but we have been
unable to show the required regularity of the limit $(S,M)$ from our
approximate system, since the time discretization is not compatible with
the technique of \cite{EZE09}.

\begin{remark}\rm
We could adapt the construction of scheme \eqref{sch.ic}--\eqref{Flux} and the proofs of our main results, Theorem \ref{thm.ex} and Theorem \ref{thm.conv}, for the approximation of the solution to a quorum-sensing-induced biofilm dispersal model introduced in \cite{ESE17}, which can be seen as a generalization of \eqref{1.eqS}--\eqref{1.F}.
\qed
\end{remark}


\section{Existence of solutions}\label{sec.ex}

For the proof of Theorem \ref{thm.ex}, we proceed by induction. 
By Hypothesis (H3), $0\le S_K^{0}\le 1$, $0\le M_K^{0}<1$ holds for $K\in\T$.
Let $(S^{k-1},M^{k-1})$ satisfy $0\le S_K^{k-1}\le 1$, $0\le M_K^{k-1}<1$
for all $K\in\T$ and some $k\in\{1,\ldots,N_T\}$. 
We use the function $Z_\eps:[0,1)\to\R$, defined by
\begin{equation}\label{3.defZ}
  Z_\eps(M) = \int_0^M F(s)ds - F(M^D)(M-M^D) + \eps\bigg(M\log\frac{M}{M^D}
	+ M^D-M\bigg), 
\end{equation}
where $\eps>0$ and $F(M)$ is given in \eqref{1.F}.

{\em Step 1: Definition of a linearized problem.} 
Let $R>0$ and set
$$
  {\mathcal K}_R := \big\{(S,W)\in\R^{2\theta}:\|S\|_{0,2,\M}<R,\ \|W\|_{1,2,\M}<R,
	\ S_\sigma=1,\ W_\sigma=0\mbox{ for }\sigma\in\E_{\rm ext}\big\},
$$
where $\theta=\#\T+\#\E_{\rm ext}$. We define the fixed-point
mapping $Q:{\mathcal K}_R\to\R^{2\theta}$ by $Q(S,W)=(S^\eps,W^\eps)$, 
where $(S^\eps,W^\eps)$ solves
\begin{align}
  & \frac{\m(K)}{\Delta t}(S^\eps_K-S_K^{k-1}) + \sum_{\sigma\in\E_K}\F_{S,K,\sigma}
	= \m(K)g([S_K]_+,M_K), \label{3.linS} \\
	& \eps\bigg(\m(K)W_K^\eps - \sum_{\sigma\in\E_K}\tau_\sigma\DD_{K,\sigma}W^\eps\bigg) 
	\label{3.linM} \\
	&\phantom{xx}{}
	= -\frac{\m(K)}{\Delta t}(M_K-M_K^{k-1}) - \sum_{\sigma\in\E_K}\F_{M,K,\sigma}
	+ \m(K)h([S_K]_+,M_K), \nonumber
\end{align}
the fluxes are as in \eqref{Flux}, $[z]_+:=\max\{0,z\}$,
and we impose the Dirichlet boundary
conditions $S_\sigma^\eps=1$, $W_\sigma^\eps=0$ for $\sigma\in\E_{\rm ext}$.
The value $M_K$ is a function of $W_K$, implicitly defined by
\begin{equation}\label{3.W}
  W_K = Z_\eps'(M_K) = F(M_K) - F(M^D) + \eps\log\frac{M_K}{M^D}, \quad K\in\T.
\end{equation}
The map $(0,1)\to\R$, $M_K\mapsto W_K$ is invertible because the function
$Z_\eps'$ is increasing. This shows that $M_K$ is well defined and $M_K\in(0,1)$
for $K\in\T$. The existence of a unique solution $(S^\eps,W^\eps)$ to
\eqref{3.linS}--\eqref{3.linM} is a consequence of \cite[Lemma 9.2]{EGH00}.

We claim that $Q$ is continuous. To show this, we first multiply \eqref{3.linM}
by $W_K^\eps$, sum over $K\in\T$, and use the discrete integration-by-parts
formula \eqref{2.ibp}:
\begin{align*}
  \eps\|W^\eps\|_{1,2,\M}^2 
	&= \eps\sum_{K\in\T}\bigg(\m(K)(W_K^\eps)^2 - \sum_{\sigma\in\E_K}\tau_\sigma
	\DD_{K,\sigma}(W^\eps) W_K^\eps\bigg)
	= J_1+J_2+J_3, \quad\mbox{where} \\
	J_1 &= -\sum_{K\in\T}\frac{\m(K)}{\Delta t}(M_K-M_K^{k-1})W_K^\eps, \\
	J_2 &= -\sum_{K\in\T}\sum_{\sigma\in\E_K}\F_{M,K,\sigma}W_K^\eps, \\
	J_3 &= \sum_{K\in\T}\m(K)\bigg(\frac{\kappa_3[S_K]_+}{\kappa_4+[S_K]_+}
	- \kappa_2\bigg)M_KW_K^\eps.
\end{align*}
By the Cauchy--Schwarz inequality and the bound $0<M_K<1$, we find that
\begin{align*}
  |J_1| &\le \frac{2}{\Delta t}\m(\Omega)^{1/2}\|W^\eps\|_{0,2,\M}, \\
	|J_2| &\le \bigg(\sum_{K\in\T}\frac{1}{\m(K)}\sum_{\sigma\in\E_K}|\F_{M,K,\sigma}|^2
	\bigg)^{1/2}\|W^\eps\|_{0,2,\M}, \\
	|J_3| &\le \bigg(\frac{\kappa_3}{\kappa_4+1}+\kappa_2\bigg)\m(\Omega)^{1/2}
	\|W^\eps\|_{0,2,\M}.
\end{align*}
Because of the assumption $\|W\|_{1,2,\M}<R$, the flux $|\F_{M,K,\sigma}|$ is 
bounded from above by a constant depending on $R$. This implies that
$|J_2|\le C(R)\|W^\eps\|_{0,2,\M}$, where $C(R)>0$ is some constant.
(Here and in the following, we denote by $C$, $C_i>0$ generic constants
whose value change from line to line.)
This shows that $\eps\|W^\eps\|_{1,2,\M}\le C(R)$ for (another)
constant $C(R)>0$. Using similar arguments, we obtain the existence of $C(R)>0$
such that $\|S^\eps\|_{0,2,\M}\le C(R)$. 

Next, let $(S_n,W_n)_{n\in\N}\subset{\mathcal K}_R$ be a sequence satisfying
$(S_n,W_n)\to (S,W)$ as $n\to\infty$. The previous uniform estimates for
$(S_n^\eps,W_n^\eps):=Q(S_n,W_n)$ show that $(S_n^\eps,W_n^\eps)$ is
bounded uniformly in $n\in\N$. 
Therefore, there exists a subsequence which is not relabeled such that
$(S_n^\eps,W_n^\eps)\to (S^\eps,W^\eps)$ as $n\to\infty$. Taking the limit 
$n\to\infty$ in \eqref{3.linS}--\eqref{3.linM}, we see that $(S^\eps,W^\eps)
= Q(S,W)$. We deduce from the uniqueness of the limit that the whole sequence
converges, which means that $Q$ is continuous.

{\em Step 2: Definition of the fixed-point operator.} 
We claim that $Q$ admits a fixed point.
We use a topological degree argument \cite[Chap.~1]{Dei85} and prove that
$\operatorname{deg}(I-Q,{\mathcal K}_R,0)=1$, where deg is the Brouwer
topological degree.  Since deg is invariant by homotopy, it is sufficient to
show that any solution $(S^\eps,W^\eps,\rho)\in \overline{{\mathcal K}}_R
\times[0,1]$ to the fixed-point equation $(S^\eps,W^\eps)=\rho Q(S^\eps,W^\eps)$
satisfies $(S^\eps,W^\eps,\rho)\not\in\pa{\mathcal K}_R\times[0,1]$ for sufficiently
large values of $R>0$. Let $(S^\eps,W^\eps,\rho)$ be a fixed point and assume
that $\rho\neq 0$, the case $\rho=0$ being clear. Then $(S^\eps,W^\eps)$ solves
\begin{align}
  & \frac{\m(K)}{\Delta t}(S^\eps_K-\rho S_K^{k-1}) 
	+ \rho\sum_{\sigma\in\E_K}\F_{S,K,\sigma}^\eps
	= \rho\m(K)g([S^\eps_K]_+,M_K^\eps), \label{3.linS2} \\
	& \eps\bigg(\m(K)W_K^\eps - \sum_{\sigma\in\E_K}\tau_\sigma\DD_{K,\sigma}W^\eps\bigg) 
	\label{3.linM2} \\
	&\phantom{xx}{}
	= -\rho\frac{\m(K)}{\Delta t}(M_K-M_K^{k-1}) 
	- \rho\sum_{\sigma\in\E_K}\F_{M,K,\sigma}
	+ \rho\m(K)h([S^\eps_K]_+,M^\eps_K) \nonumber
\end{align}
for $K\in\T$ with the boundary conditions $S_\sigma^\eps=1$, $W_\sigma^\eps=0$
for $\sigma\in\E_K$, the fluxes are given by \eqref{Flux} with $(S,M)$ replaced
by $(S^\eps,M^\eps)$, and $M^\eps_K$ is the unique solution to \eqref{3.W} with
$W_K$ replaced by $W_K^\eps$. 

{\em Step 3: A priori estimates.}
We establish some a priori estimates for the fixed points $(S^\eps,W^\eps)$ 
of $Q$, which are uniform in $R$. Definition \eqref{3.W} immediately gives the
bound $0<M_K^\eps<1$ for all $K\in\T$.

\begin{lemma}[Pointwise bounds for $S^\eps$]\label{lem.S}
The following bounds hold:
$$
  0\le S_K^\eps\le 1 \quad\mbox{for }K\in\T.
$$
\end{lemma}

\begin{proof}
First, we multiply \eqref{3.linS2} by $\Delta t[S_K^\eps]_-$, where
$[z]_-=\min\{0,z\}$, and sum over $K\in\T$. Then, after a discrete integration
by parts,
$$
  \sum_{K\in\T}\m(K)[S_K^\eps]_-^2 + \rho d_1\Delta t\sum_{\sigma\in\E}
	\tau_\sigma\DD_{K,\sigma}(S^\eps)\DD_{K,\sigma}[S^\eps]_-
	= \rho\sum_{K\in\T}\m(K)S_K^{k-1}[S_K^\eps]_- \le 0,
$$
since $g([S_K^\eps]_+,M_K^\eps)[S_K^\eps]_-=0$ and $S_K^{k-1}\ge 0$ by
the induction hypothesis. The second term on the left-hand side is nonnegative,
since $z\mapsto [z]_-$ is monotone.
This implies that the first term must be nonpositive, showing that
$[S_K^\eps]_-=0$ and hence $S_K^\eps\ge 0$ for all $K\in\T$. 

To verify the upper bound for $S^\eps$, we multiply \eqref{3.linS2} by
$\Delta t[S_K^\eps-1]_+$, sum over $K\in\T$, and use discrete integration by parts:
\begin{align}
  \sum_{K\in\T}&\m(K)\big((S_K^\eps-1) - (\rho S_K^{k-1}-1)\big)[S_K^\eps-1]_+
	+ \rho d_1\Delta t\sum_{\sigma\in\E}\DD_{K,\sigma}(S^\eps-1)
	\DD_{K,\sigma}[S^\eps-1]_+ \nonumber \\
	&= \rho\Delta t\sum_{K\in\T}\m(K)g\big(S_K^\eps,M_K^\eps\big)[S_K^\eps-1]_+
	\le 0, \label{3.aux}
\end{align}
since we have always $g(S_K^\eps,M_K^\eps)\le 0$.
It follows from the induction hypothesis and $\rho\le 1$
that $\rho S_K^{k-1}\le 1$, and the first term on the left-hand side
can be estimated according to
$$
  \sum_{K\in\T}\m(K)\big((S_K^\eps-1) - (\rho S_K^{k-1}-1)\big)[S_K^\eps-1]_+
	\ge \sum_{K\in\T}\m(K)[S_K^\eps-1]_+^2.
$$
We deduce from the monotonicity of $z\mapsto[z]_+$ that 
the second term on the left-hand side of \eqref{3.aux} is nonnegative as well.
Hence,
$\sum_{K\in\T}\m(K)[S_K^\eps-1]_+^2\le 0$ and consequently $S_K^\eps\le 1$
for all $K\in\T$.
\end{proof}

\begin{lemma}[Estimate for $F(M_K^\eps)$]\label{lem.grad}
There exist constants $C_1$, $C_2>0$, only depending on the given data, 
such that
\begin{align}
  \eps\Delta t&\|W^\eps\|_{1,2,\M}^2 + \rho\|Z(M^\eps)\|_{0,1,\M}
	+ \rho\Delta t C_1\|F(M^\eps)-F(M^D)\|_{1,2,\M}^2 \label{3.Z} \\
	&\le \Delta t C_2 + \|Z_\eps(M^{k-1})\|_{0,1,\M}. \nonumber
\end{align}
\end{lemma}

\begin{proof}
We multiply \eqref{3.linM2} by $\Delta t W_K^\eps$, sum over $K$, and use discrete
integration by parts:
\begin{align}
  & \eps\Delta t\|W^\eps\|_{1,2,\M}^2 + J_4 + J_5 = J_6, \quad\mbox{where} 
	\label{3.aux2} \\
	& J_4 = \rho\sum_{K\in\T}\m(K)(M_K^\eps-M_K^{k-1})W_K^\eps, \nonumber \\
	& J_5 = \rho\Delta t d_2\sum_{\sigma\in\E}\tau_\sigma\DD_{K,\sigma}F(M^\eps)
	\DD_{K,\sigma}W^\eps, \nonumber \\
	& J_6 = \rho\Delta t\sum_{K\in\T}\m(K)h(S_K^\eps,M_K^\eps)W^\eps. \nonumber
\end{align}
By the convexity of $Z_\eps$, $(M_K^\eps-M_K^{k-1})Z'_\eps(M_K^\eps)
\ge Z(M_K^\eps)-Z_\eps(M_K^{k-1})$ such that
\begin{align*}
  J_4 &\ge \rho\sum_{K\in\T}\m(K)\bigg\{Z(M_K^\eps) 
	+ \eps\bigg(M_K^\eps \log\frac{M_K^\eps}{M^D} + M^D - M_K^\eps\bigg)
	- Z_\eps(M_K^{k-1})\bigg\} \\
	&\ge \rho\|Z(M_K^\eps)\|_{0,1,\M} - \rho\|Z_\eps(M_K^{k-1})\|_{0,1\M}.
\end{align*}
The definition of $W^\eps_K$ and the monotonicity of the functions $F$ and $\log$
imply that
\begin{align}
  J_5 &= \rho\Delta td_2\sum_{K\in\T}\m(K)\big([\DD_{K,\sigma}(F(M^\eps)-F(M^D))]^2
	+ \eps\DD_{K,\sigma}F(M^\eps)\DD_{K,\sigma}\log M^\eps\big) \label{3.J5} \\
	&\ge \rho\Delta td_2|F(M^\eps)-F(M^D)|_{1,2,\M}^2
	\ge \rho\Delta t d_2C(\xi)\|F(M^\eps)-F(M^D)\|_{1,2\M}^2, \nonumber
\end{align}
where the last step follows from the discrete Poincar\'e inequality
\cite[Theorem 3.2]{BCF15}. Finally, by the Young inequality and taking into account
the bounds $S_K^\eps\le 1$ and $M_K^\eps<1$, we find that
\begin{align*}
  J_6 &\le \rho\Delta t
	\bigg(\kappa_2+\frac{\kappa_3}{\kappa_4+1}\bigg)\sum_{K\in\T}\m(K)
	\bigg(|F(M_K^\eps)-F(M^D)| + \eps M_K^\eps\bigg|\log\frac{M_K^\eps}{M^D}\bigg|
	\bigg) \\
	&\le \frac{\eta}{2}\rho\Delta t\bigg(\kappa_2+\frac{\kappa_3}{\kappa_4+1}\bigg)
	\|F(M^\eps)-F(M^D)\|_{1,2,\M}^2 + \frac{\Delta t}{2\eta}
	\bigg(\kappa_2+\frac{\kappa_3}{\kappa_4+1}\bigg)\m(\Omega) \\
	&\phantom{xx}{}+ \eps\Delta t C(\Omega),
\end{align*}
where $\eta>0$. 
Inserting the estimates for $J_4$, $J_5$, and $J_6$ into \eqref{3.aux2} yields
\begin{align*}
  \eps\Delta t\|W^\eps\|_{1,2,\M}^2 
	&+ \rho\Delta t\bigg(d_2 C(\xi)
	- \frac{\eta}{2}\bigg(\kappa_2+\frac{\kappa_3}{\kappa_4+1}
	\bigg)\bigg)\|F(M^\eps)-F(M^D)\|_{1,2,\M}^2 \\
	&{}+ \rho\|Z(M^\eps)\|_{0,1,\M} 
	\le \rho\|Z_\eps(M^{k-1})\|_{0,1,\M} + \Delta t C(\eta).
\end{align*}
Then, choosing $\eta>0$ sufficiently small shows the conclusion.
\end{proof}

{\em Step 4: Topological degree argument.}
We deduce from the estimates of Lemmas \ref{lem.S}--\ref{lem.grad} that
$$
  \|S^\eps\|_{0,2,\M} \le \m(\Omega)^{1/2}, \quad
	\|W^\eps\|_{1,2,\M} \le \frac{1}{\sqrt{\eps\Delta t}}
	(\|Z_\eps(M^{k-1})\|_{0,1,\M} + \Delta tC)^{1/2}.
$$
Thus, choosing
$$
  R = \max\bigg\{\m(\Omega)^{1/2},\frac{1}{\sqrt{\eps\Delta t}}
	(\|Z_\eps(M^{k-1})\|_{0,1,\M} + \Delta tC)^{1/2}\bigg\} + 1,
$$
we see that $(S^\eps,W^\eps)\not\in\pa{\mathcal K}_R$ and $\operatorname{deg}
(I-Q,{\mathcal K}_R,0)=1$. We conclude that $Q$ admits a fixed point, i.e.\
a solution $(S^\eps,W^\eps)$ to \eqref{3.linS2}--\eqref{3.linM2}.

{\em Step 5: Limit $\eps\to 0$.} Thanks to Lemmas \ref{lem.S}--\ref{lem.grad}
and the bound $0<M_K^\eps<1$, there exist subsequences, which are not relabeled,
such that $S_K^\eps\to S_K^k$, $M_K^\eps\to M_K^k$, and $\eps W_K^\eps\to 0$ as
$\eps\to 0$, where $0\le S_K^k\le 1$ and $0\le M_K^k\le 1$ for all $K\in\T$.
Passing to the limit $\eps\to 0$ in \eqref{3.Z} and taking into account the
lower semicontinuity of $F$, we find that
$$
  \Delta t C_1\|F(M^k)-F(M^D)\|_{0,2,\M}^2 \le \|Z(M^{k-1})\|_{0,1,\M} + \Delta tC
	< \infty.
$$
Thus, $F(M_K^k)$ is finite, which implies that $M_K^k<1$ for any $K\in\T$.
We can perform the limit $\eps\to 0$ in \eqref{3.linS2}--\eqref{3.linM2} to
deduce the existence of a solution $(S^k,M^k)$ to scheme \eqref{sch.ic}--\eqref{Flux}.

{\em Step 6: Positive lower bound for $M^k$.}
Again, we proceed by induction.
Let $M^0\ge m_0$ in $\Omega$ and $M^D\ge m_0$. Then $M_K^0\ge m_0$ for all $K\in\T$.
Set $m^k=m_0(1+\kappa_2\Delta t)^{-k}$. 
The induction hypothesis reads as $M_K^{k-1}\ge m^{k-1}$ for $K\in\T$. We multiply
\eqref{schM} by $\Delta t[M_K^k-m^k]_-$, sum over $K\in\T$, and use discrete
integration by parts:
\begin{align*}
  & \sum_{K\in\T}\m(K)(M_K^k-M_K^{k-1})[M_K^k-m^k]_- = J_7 + J_8, \quad\mbox{where} \\
	& J_7 = -\Delta t\sum_{\sigma\in\E}\tau_\sigma\DD_{K,\sigma}F(M^k)
	\DD_{K,\sigma}[M_K^k-m^k]_-, \\
	& J_8 = \Delta t\sum_{K\in\T}\m(K)h(S^k,M^k)[M_K^k-m^k]_-.
\end{align*}
Taking into account that $M_K^{k-1}-m^{k-1}\ge 0$ and 
$m^k-m^{k-1}=-\kappa_2\Delta tm^k$, we estimate the left-hand side according to
\begin{align*}
  \sum_{K\in\T}&\m(K)(M_K^k-M_K^{k-1})[M_K^k-m^k]_- \\
	&= \sum_{K\in\T}\m(K)\big((M_K^k-m^k)-(M_K^{k-1}-m^{k-1})\big)[M_K^k-m^k]_- \\
	&\phantom{xx}{}+ \sum_{K\in\T}\m(K)(m^k-m^{k-1})[M_K^k-m^k]_- \\
	&\ge \sum_{K\in\T}\m(K)[M_K^k-m^k]_-^2 
	- \kappa_2\Delta t m^k\sum_{K\in\T}\m(K)[M_K^k-m^k]_-.
\end{align*}
Since $F$ and $z\mapsto [z-m^k]_-$ are monotone, we have $J_7\le 0$. Furthermore,
\begin{align*}
  J_8 &= \Delta t\sum_{K\in\T}\m(K)
	\bigg(\frac{\kappa_3 S_K^k}{\kappa_4+S_K^k}-\kappa_2\bigg)M_K^k[M_K^k-m^k]_- \\
	&\le -\kappa_2\Delta t\sum_{K\in\T}\m(K)M_K^k[M_K^k-m^k]_-
	\le -\kappa_2\Delta t\sum_{K\in\T}\m(K)m^k[M_K^k-m^k]_-.
\end{align*}
The terms involving $\kappa_2$ cancel and we end up with
$$
	\sum_{K\in\T}\m(K)[M_K^k-m^k]_-^2 \le 0.
$$
It follows that
$[M_K^k-m^k]_-=0$ and hence $M_K^k\ge m^k\ge m_0\exp(-\kappa_2 k\Delta t)$.


\section{Uniqueness of solutions}\label{sec.unique}

We proceed by induction. Let $k\in\{1,\ldots,N_T\}$, let
$(S_{1}^k,M_{1}^k)$ and $(S_2^k,M_2^k)$ be two solutions to scheme
\eqref{sch.ic}--\eqref{Flux}, and assume that $S_1^{k-1}=S_2^{k-1}$,
$M_1^{k-1}=M_2^{k-1}$. We wish to show that $S_1^k=S_2^k$, $M_1^k=M_2^k$.
The functions $S_{1}^k-S_{2}^k$ and $M_{1}^k-M_{2}^k$
are solutions, respectively, to
\begin{align}
  \frac{\m(K)}{\Delta t}(S_{1,K}^k-S_{2,K}^k) - d_1\sum_{\sigma\in\E_K}\tau_\sigma
	\DD_{K,\sigma}(S_1^k-S_2^k)
	&= \m(K)G_K^k, \label{3.S} \\
	\frac{\m(K)}{\Delta t}(M_{1,K}^k-M_{2,K}^k) - d_2\sum_{\sigma\in\E_K}\tau_\sigma
	\DD_{K,\sigma}(F(M_1^k)-F(M_2^k)) &= \m(K)H_K^k \label{3.M}
\end{align}
for $K\in\T$, where
\begin{align*}
  G_K^k &= -\frac{\kappa_1 S_{1,K}^k}{\kappa_4+S_{1,K}^k}(M_{1,K}^k-M_{2,K}^k)
	- \frac{\kappa_1\kappa_4 M_{2,K}^k}{(\kappa_4+S_{1,K}^k)(\kappa_4+S_{2,K}^k)}
	(S_{1,K}^k-S_{2,K}^k), \\
	H_K^k &= \bigg(\frac{\kappa_3 S_{1,K}^k}{\kappa_4+S_{1,K}^k} - \kappa_2\bigg)
	(M_{1,K}^k-M_{2,K}^k) + \frac{\kappa_3\kappa_4 M_{2,K}^k}{(\kappa_4+S_{1,K}^k)
	(\kappa_4+S_{2,K}^k)}(S_{1,K}^k-S_{2,K}^k).
\end{align*}

Now, let the vectors $(\psi^k_\T,\psi^k_\E)$ and $(\phi^k_\T,\phi^k_\E)$ be the unique
solutions to
\begin{align*}
  -\sum_{\sigma\in\E_K}\tau_\sigma\DD_{K,\sigma}\psi^k 
	&= \m(K)(S_{1,K}^k-S_{2,K}^k), \\
  -\sum_{\sigma\in\E_K}\tau_\sigma\DD_{K,\sigma}\phi^k 
	&= \m(K)(M_{1,K}^k-M_{2,K}^k)
\end{align*}
for $K\in\T$, where we impose the boundary conditions $\psi^k_\sigma=\phi^k_\sigma=0$
for $\sigma\in\E_{\rm ext}$. The existence and uniqueness of these solutions
is a direct consequence of \cite[Lemma 9.2]{EGH00}. We multiply \eqref{3.M} by
$\phi_K^k$ and sum over $K\in\T$:
\begin{align}
  & \frac{1}{\Delta t}\sum_{K\in\T}\m(K)(M_{1,K}^k-M_{2,K}^k)\phi^k_K 
	= I_1 + I_2, \quad\mbox{where} \label{3.eqM} \\
	& I_1 = d_2\sum_{K\in\T}\sum_{\sigma\in\E_K}\tau_\sigma\DD_{K,\sigma}
	(F(M_{1,K}^k)-F(M_{2,K}^k))\phi^k_K, \quad
	I_2 = \sum_{K\in\T}\m(K)H_K^k\phi^k_K. \nonumber
\end{align}
Inserting the equation for $\phi^k$ and using discrete integration by parts gives
$$
  \sum_{K\in\T}\m(K)(M_{1,K}^k-M_{2,K}^k)\phi^k_K
	= -\sum_{K\in\T}\sum_{\sigma\in\E_K}\tau_\sigma\DD_{K,\sigma}(\phi^k)\phi^k_K
	= \sum_{\sigma\in\E}\tau_\sigma(\DD_{K,\sigma}\phi^k)^2 = |\phi^k|_{1,2,\M}^2.
$$
Concerning the sum $I_1$, we use the equation for $\phi^k$ again,
apply discrete integration by parts twice, and take into account the positive 
lower bound for $M_{i}^k$ from Theorem \ref{thm.ex}:
\begin{align*}
  I_1 &= d_2\sum_{K\in\T}(F( M_{1,K}^k)-F(M_{2,K}^k))\sum_{\sigma\in\E_K}
	\tau_\sigma\DD_{K,\sigma}\phi^k \\
	&= -d_2\sum_{K\in\T}\m(K)(F(M_{1,K}^k)-F(M_{2,K} ^k))(M_{1,K}^k-M_{2,K}^k) \\
	&\le -d_2 c_0\sum_{K\in\T}\m(K)(M_{1,K}^{k}-M_{2,K}^k)^2,
\end{align*}
where $c_0>0$ depends on the minimum of $M_1^k$ or $M_2^k$. Finally, because
of the bounds $0\le S_K^k\le 1$ and $0\le M_K^k < 1$ from Theorem \ref{thm.ex}, the
Young inequality and the discrete Poincar\'e inequality \cite[Theorem 3.2]{BCF15},
\begin{align*}
  I_2 &\le -\kappa_2|\phi^k|_{1,2,\M}^2 + \sum_{K\in\T}\m(K)
	\bigg(\frac{\kappa_3}{\kappa_4+1}|M_{1,K}^k-M_{2,K}^k|
	+ \frac{\kappa_3}{\kappa_4}|S_{1,K}^k-S_{2,K}^k|\bigg)|\phi^k_K| \\
	&\le 
	\frac{\delta}{2}\bigg(\frac{\kappa_3^2}{(\kappa_4+1)^2}
	\|M_1^k-M_2^k\|_{0,2,\M}^2 + \frac{\kappa_3^2}{\kappa_4^2}\|S_1^k-S_2^k\|_{0,2,\M}^2 
	\bigg) + \frac{C}{\delta\xi}|\phi^k|_{1,2,\M}^2,
\end{align*}
where $\delta>0$ is arbitrary. 
Collecting these estimates, we infer from \eqref{3.eqM} that
\begin{align*}
  \bigg(\frac{1}{\Delta t} &- \frac{C}{\delta\xi}\bigg)|\phi^k|_{1,2,\M}^2 
	+ \frac12d_2 c_0\|M_1^k-M_2^k\|_{0,2,\M}^2 \\
	&\le \frac{\delta}{2}\bigg(\frac{\kappa_3^2}{(\kappa_4+1)^2}
	\|M_1^k-M_2^k\|_{0,2,\M}^2 
	+ \frac{\kappa_3^2}{\kappa_4^2}\|S_1^k-S_2^k\|_{0,2,\M}^2\bigg).
\end{align*}
Arguing similarly for equation \eqref{3.S}, we arrive to
\begin{align*}
   \bigg(\frac{1}{\Delta t} &- \frac{C}{\delta\xi}\bigg)|\psi^k|_{1,2,\M}^2 
	+ \frac12d_1\|S_1^k-S_2^k\|_{0,2,\M}^2 \\
	&\le \frac{\delta}{2}\bigg(\frac{\kappa_1^2}{(\kappa_4+1)^2}
	\|M_1^k-M_2^k\|_{0,2,\M}^2 
	+ \frac{\kappa_1^2}{\kappa_4^2}\|S_1^k-S_2^k\|_{0,2,\M}^2\bigg).
\end{align*}
We set $R^k:=\|S_1^k-S_2^k\|_{0,2,\M}^2+\|M_1^k-M_2^k\|_{0,2,\M}^2$.
Then an addition of the previous two inequalities yields
$$
  \bigg(\frac{1}{\Delta t} - \frac{C}{\delta\xi}\bigg)\big(|\phi^k|_{1,2,\M}^2 
	+ |\psi^k|_{1,2,\M}^2\big)
	+ \frac12\bigg(\min\{d_1,d_2c_0\} - \delta\frac{\kappa_1^2+\kappa_3^2}{\kappa_4^2}
	\bigg)R^k \le 0.
$$
Choosing $\delta\le \kappa_4^2/(\kappa_1^2+\kappa_3^2)\min\{d_1,d_2c_0\}$ and
$\Delta t < C/(\delta\xi)$, both terms are nonnegative, and we infer that
$\phi^k_K=\psi_K^k=0$ and consequently $M_{1,K}^k-M_{2,K}^k=S_{1,K}^k-S_{2,K}^k=0$
for all $K\in\T$.


\section{Uniform estimates}\label{sec.est}

We establish some estimates that are uniform with respect to $\Delta x$ and $\Delta t$.
The first bounds follow from the results of the previous section.

\begin{lemma}[Uniform estimates I]\label{lem.unif1}
There exists a constant $C>0$ independent of $\Delta x$ and $\Delta t$ such that
\begin{align*}
  & 0\le S_K^k\le 1, \quad 0\le M_K^k < 1\quad\mbox{for }K\in\T, \\
  & \sum_{k=1}^{N_T}\Delta t\big(\|F(M^k)\|_{1,2,\M}^2 + \|S^k\|_{1,2,\M}^2\big) \le C.
\end{align*}
\end{lemma}

\begin{proof}
The $L^\infty$ bounds follow directly from Theorem \ref{thm.ex}, while
the discrete gradient bound for $F(M^k)$ is a consequence of Lemma \ref{lem.grad}.
It remains to show the discrete gradient bound for $S^k$. We multiply \eqref{schS}
by $\Delta t(S_K^k-1)$, sum over $K\in\T$, and use discrete integration by parts:
\begin{align}
  \sum_{K\in\T}&\m(K)(S_K^k-S_K^{k-1})(S_K^k-1) 
	= -\Delta t\sum_{\sigma\in\E}\tau_\sigma\DD_{K,\sigma}(S^k)\DD_{K,\sigma}(S^k-1) 
	\label{2.aux2} \\
	&\phantom{xx}{}+ \Delta t\sum_{K\in\T}\m(K)g(S_K^k,M_K^k)(S_K^k-1) \nonumber \\
  &\le -\Delta t\sum_{\sigma\in\E}\tau_\sigma(\DD_{K,\sigma}(S^k-1))^2
	+ \Delta t\sum_{K\in\T}\m(K)\frac{\kappa_1 S_K^kM_K^k}{\kappa_4+S_K^k}. \nonumber
\end{align}
The left-hand side is bounded from below by
$$
  \sum_{K\in\T}\m(K)\big((S_K^k-1)-(S_K^{k-1}-1)\big)(S_K^k-1)
	\ge \frac12\sum_{K\in\T}\m(K)\big((S_K^k-1)^2 - (S_K^{k-1}-1)^2\big).
$$
In view of the upper bounds for $S_K^k$ and $M_K^k$, the last term on the right-hand
side of \eqref{2.aux2} is bounded by $\Delta t\m(\Omega)\kappa_1/(\kappa_4+1)$.
Therefore, it follows from \eqref{2.aux2} that
$$
  \frac12\sum_{K\in\T}\m(K)(S_K^k-1)^2 
	+ \Delta t|S_K^k-1|_{1,2,\M}^2
	\le \frac12\sum_{K\in\T}\m(K)(S_K^{k-1}-1)^2 + C\Delta t.
$$
Summing this inequality from $k=1,\ldots,N_T$, we find that
$$
  \frac12\|S^{N_T}-1\|_{0,2,\M}^2 + \sum_{k=1}^{N_T}\Delta t|S_K^k-1|_{1,2,\M}^2
	\le \frac12\|S^0-1\|_{0,2,\M}^2 + CT.
$$
This yields the desired estimate.
\end{proof}

We also need an estimate for the time translates of the solution. For this,
let $\phi\in C_0^\infty(\Omega_T)$ be given and define 
$\phi^k=(\phi^k_\T,\phi^k_\E)\in\R^\theta$ (recall that $\theta=\#\T+\#\E$)
for $k=1,\ldots,N_T$ by
$$
  \phi_K^k = \frac{1}{\m(K)}\int_K\phi(x,t_k)dx, \quad
	\phi_\sigma^k = \frac{1}{\m(\sigma)}\int_\sigma\phi(s,t_k)ds =0,
$$
where $K\in\T$ and $\sigma\in\E_{\rm ext}$. 

\begin{lemma}[Uniform estimates II]\label{lem.unif2}
For any $\phi\in C_0^\infty(\Omega_T)$,
there exist constants $C_3$, $C_4>0$, only depending on the data and the mesh, 
such that
\begin{align*}
  \sum_{k=1}^{N_T}\sum_{K\in\T}\m(K)(M_K^k-M_K^{k-1})\phi_K^k
	&\le C_3\|\na\phi\|_{L^\infty(\Omega_T)}, \\
	\sum_{k=1}^{N_T}\sum_{K\in\T}\m(K)(S_K^k-S_K^{k-1})\phi_K^k
	&\le C_4\|\na\phi\|_{L^\infty(\Omega_T)}.
\end{align*}
\end{lemma}

\begin{proof}
We multiply \eqref{schM} by $\Delta t\phi_K^k$, sum over $K\in\T$ and
$k=1,\ldots,N_T$, and use discrete integration by parts. Then
\begin{align}
  & \sum_{k=1}^{N_T}\sum_{K\in\T}\m(K)(M_K^k-M_K^{k-1})\phi_K^k = I_3+I_4, \quad
	\mbox{where} \label{6.aux} \\
	& I_3 = -d_2\sum_{k=1}^{N_T}\Delta t\sum_{\sigma\in\E}\tau_\sigma
	\DD_{K,\sigma}F(M^k)\DD_{K,\sigma}\phi^k, \nonumber \\
	& I_4 = \sum_{k=1}^{N_T}\Delta t\sum_{K\in\T}\m(K)
	\bigg(\frac{\kappa_3 S_K^k}{\kappa_4+S_K^k}-\kappa_2\bigg)M_K^k\phi_K^k.
	\nonumber
\end{align}
It follows from the Cauchy--Schwarz inequality, Lemma \ref{lem.unif1}, and the mesh
regularity \eqref{2.regmesh} that
\begin{align*}
  |I_3| &\le d_2 C\|\na\phi\|_{L^\infty(\Omega_T)}\bigg(\sum_{k=1}^{N_T}\Delta t
	\sum_{K\in\T}\sum_{\sigma\in\E_K}\m(\sigma)\dist_\sigma\bigg)^{1/2} \\
	&\le d_2 C\xi^{-1/2}\|\na\phi\|_{L^\infty(\Omega_T)}\bigg(\sum_{k=1}^{N_T}\Delta t
	\sum_{K\in\T}\sum_{\sigma\in\E_K}\m(\sigma)\dist(x_K,\sigma)\bigg)^{1/2} \\
  &= d_2C\sqrt{2\m(\Omega)T\xi^{-1}}\|\na\phi\|_{L^\infty(\Omega_T)},
\end{align*}
where we used \eqref{2.estmesh} in the last step. Next, using similar 
arguments and the discrete Poincar\'e inequality \cite[Theorem 3.2]{BCF15},
\begin{align*}
  |I_4| &\le \bigg(\kappa_2+\frac{\kappa_3}{\kappa_4+1}\bigg)\sqrt{T\m(\Omega)}
	\bigg(\sum_{k=1}^{N_T}\Delta t\|\phi^k\|_{0,2,\M}^2\bigg)^{1/2} \\
	&\le \bigg(\kappa_2+\frac{\kappa_3}{\kappa_4+1}\bigg)\sqrt{T\m(\Omega)C\xi^{-1}}
	\bigg(\sum_{k=1}^{N_T}\Delta t|\phi^k|_{1,2,\M}^2\bigg)^{1/2} \\
	&\le \bigg(\kappa_2+\frac{\kappa_3}{\kappa_4+1}\bigg)\sqrt{T\m(\Omega)C\xi^{-1}}
	\|\na\phi\|_{L^\infty(\Omega_T)}\bigg(\sum_{k=1}^{N_T}\Delta t\sum_{K\in\T}
	\sum_{\sigma\in\E_K}\m(\sigma)\dist_\sigma\bigg)^{1/2} \\
	&\le C(T,\Omega,\xi)\xi^{-1}\bigg(\kappa_2+\frac{\kappa_3}{\kappa_4+1}\bigg)
	\|\na\phi\|_{L^\infty(\Omega_T)}. 
\end{align*}
Inserting these estimates into \eqref{6.aux} shows the first statement of the
lemma. The second statement is proved in a similar way.
\end{proof}


\section{Convergence of the scheme}\label{sec.conv}

The compactness follows from the uniform estimates proved in the previous
section and the discrete compensated compactness result 
obtained in \cite[Theorem 3.9]{ACM17}. 

\begin{lemma}[Compactness]\label{lem.comp}
Let $(S_m,M_m)_{m\in\N}$ be a sequence of solutions to scheme 
\eqref{sch.ic}--\eqref{Flux} constructed in Theorem \ref{thm.ex}.
Then there exists $(S,M)\in L^\infty(\Omega_T;\R^2)$ satisfying $F(M)$, $S\in
L^2(0,T;H^1(\Omega))$ such that, up to a subsequence, as $m\to\infty$,
\begin{align*}
  M_m\to M, \quad S_m\to S &\quad\mbox{a.e. in }\Omega_T, \\
  F(M_m)\to F(M) &\quad\mbox{strongly in }L^r(\Omega_T)\mbox{ for }
	1 \le r < 2, \\
	\na^m F(M_m)\rightharpoonup \na F(M), \quad \na^m S_m\rightharpoonup\na S
	&\quad\mbox{weakly in }L^2(\Omega_T).
\end{align*}
\end{lemma}

\begin{proof}
The a.e.\ convergence for $M_m$ is a consequence of \cite[Theorem 3.9]{ACM17}. Indeed,
the estimates in Lemmas \ref{lem.unif1}--\ref{lem.unif2} correspond to
conditions (a)--(c) in \cite[Prop.~3.8]{ACM17}, while assumptions
(A$_{\rm t}$1), (A$_{\rm x}$1)--(A$_{\rm x}$3) are satisfied 
for our implicit Euler finite-volume scheme.
We infer that there exists a subsequence which is not
relabeled such that $M_m\to M$ and $F(M_m)\to F(M)$ a.e.\ in $\Omega_T$. 
In view of Lemma \ref{lem.unif1}, the sequence $(F(M_m))$ is bounded in 
$L^2(\Omega_T)$, and thanks to the Vitali's lemma, we conclude that 
$F(M_m)\to F(M)$ strongly in $L^r(\Omega_T)$ for all $1 \le r < 2$.

As a consequence of the gradient estimate in Lemma \ref{lem.grad},
there exists a subsequence of $(\na^m F(M_m))$ (not relabeled) such that
$\na^m F(M_m)\rightharpoonup \Psi$ weakly in $L^2(\Omega_T)$ as $m\to\infty$.
The limit $\Psi$ can be identified with $F(M)$ by following the arguments
in the proof of \cite[Lemma 4.4]{CLP03}. Indeed, the idea is to prove that 
for all $\phi\in C_0^\infty(\Omega_T;\R^2)$,
$$
  A_m := \int_0^T\int_\Omega\na^m F(M_m)\cdot\phi dxdt
	+ \int_0^T\int_\Omega F(M_m)\diver\phi dxdt\to 0
$$
as $m\to\infty$. This is done by reformulating the two integrals: 
\begin{align*}
  \int_\Omega\na^m F(M_m)\cdot\phi dx
	&= -\frac12\sum_{K\in\T}\sum_{\sigma\in\E_{{\rm int},K}}
	\frac{\m(\sigma)}{\m(T_{K,\sigma})}\DD_{K,\sigma}F(M_{m})
	\int_{T_{K,\sigma}}\phi(s,t)\cdot\nu_{K,\sigma}dx, \\
	\int_\Omega F(M_m)\diver\phi dx &= \frac12\sum_{K\in\T}
	\sum_{\sigma\in\E_{{\rm int},K}}\DD_{K,\sigma}F(M_m)\int_\sigma\phi(s,t)
	\cdot\nu_{K,\sigma}ds.
\end{align*}
Because of the property (see \cite[Lemma 4.4]{CLP03})
$$
  \bigg|\frac{1}{\m(T_{K,\sigma})}\int_{T_{K,\sigma}}
	\phi(t,s)\cdot\nu_{K,\sigma}dx - \frac{1}{\m(\sigma)}\int_\sigma\phi(s,t)
	\cdot\nu_{K,\sigma}ds\bigg|
	\le \eta_m\|\phi\|_{C^1(\overline{\Omega})}
$$
and the uniform estimates for $F(M_{m})$ from Lemma \ref{lem.unif1}, 
it follows that 
\begin{align*}
  |A_m| &\le \frac12\sum_{k=1}^{N_T}\Delta t_m\sum_{K\in\T}
	\sum_{\sigma\in\E_{{\rm int},K}}\m(\sigma)\DD_{K,\sigma}F(M^k) \\
	&\phantom{xx}{}\times\bigg|\frac{1}{\m(T_{K,\sigma})}\int_{T_{K,\sigma}}
	\phi(t,s)\cdot\nu_{K,\sigma}dx - \frac{1}{\m(\sigma)}\int_\sigma\phi(s,t)
	\cdot\nu_{K,\sigma}ds\bigg| \\
  &\le \eta_m C\|\phi\|_{C^1(\overline{\Omega})}\to 0\quad\mbox{as }m\to\infty.
\end{align*}
This implies that $\Psi=\na F(M)$.
Finally, similar arguments as above show the convergence results for $S_m$
and $\na^m S_m$.
\end{proof}

\begin{lemma}[Convergence of the traces]\label{lem.trace}
Let $(S_m,M_m)_{m\in\N}$ be a sequence of solutions to scheme 
\eqref{sch.ic}--\eqref{Flux} constructed in Theorem \ref{thm.ex}.
Then the limit function $(S,M)$ obtained in Lemma \ref{lem.comp} satisfies
$$
  S-1, \quad F(M)-F(M^D)\in L^2(0,T;H_0^1(\Omega)).
$$
\end{lemma}

\begin{proof}
The proof for $S$ is a direct consequence of \cite[Prop.~4.9]{BCH13}. For $F(M)$, 
we follow the proof of \cite[Prop.~4.11]{BCH13}. In particular, we aim to prove that 
\begin{equation}\label{6.trace1}
  \int_0^T \int_{\pa\Omega} (F(M_m) - F(M))\psi dxdt \to 0 \quad 
	\mbox{as }m \to \infty
\end{equation}
for every $\psi \in C^\infty_0(\pa\Omega\times(0,T))$.
If this result holds then, as $M_m = M^D$ on $\pa\Omega\times(0,T)$, 
we obtain 
\begin{align*}
  \int_0^T \int_{\pa\Omega} (F(M)-F(M^D)) \psi dx dt 
	&= \lim_{m \to \infty} \bigg( \int_0^T \int_{\pa\Omega} (F(M)-F(M_m)) \psi dxdt \\
  &\phantom{xx}{}+ \int_0^T \int_{\pa\Omega} (F(M_m)-F(M^D))\psi dxdt\bigg) = 0,
\end{align*}
which implies that $F(M)=F(M^D)$ a.e.\ on $\pa\Omega\times(0,T)$. 

To prove \eqref{6.trace1}, we choose a fixed $m\in\N$ and
introduce another definition of the trace of $M_m$, 
denoted by $\widetilde{M}_{m}$, such that $\widetilde{M}_m(x,t)= M^k_K$ if 
$(x,t) \in \sigma \times (t_{k-1},t_k]$ with $\sigma \in \E_{\rm{ext},K}$. 
Following \cite{BCH13}, we notice that the property \eqref{6.trace1} is equivalent to
\begin{equation}\label{6.trace2}
  \int_0^T \int_{\pa\Omega} (F(\widetilde{M}_m) - F(M)) \psi dxdt \to 0 
	\quad \mbox{as }m \to \infty
\end{equation}
for all $\psi \in C^\infty_0(\pa\Omega\times(0,T))$. Indeed, we have, by the
Cauchy--Schwarz inequality,
\begin{align*}
  \int_0^T \int_{\pa\Omega} |F(M_m)-F(\widetilde{M}_m)| dx dt 
	&= \sum_{k=1}^{N_T} \Delta t_m \sum_{K \in \T} \sum_{\sigma \in \E_{\mathrm{ext},K}} 
	\m(\sigma) |F(M^D)-F(M^k_K)| \\
  &\le \bigg(\sum_{k=1}^{N_T}\Delta t_m\sum_{K\in\T}\sum_{\sigma\in\E_{\mathrm{ext},K}} 
	\tau_\sigma |F(M^D)-F(M^k_K)|^2\bigg)^{1/2} \\
  &\phantom{xx}{}\times\bigg(\sum_{k=1}^{N_T} \Delta t_m \sum_{K\in\T}
	\sum_{\sigma\in\E_{\mathrm{ext},K}}\m(\sigma)\dist_\sigma \bigg)^{1/2}.
\end{align*}
Hence, thanks to Lemma \ref{lem.unif1} and the fact that $\dist_\sigma 
= \dist(x_K,\sigma) \le\mathrm{diam}(K) \le\eta_m$ for every 
$\sigma \in \E_{\mathrm{ext},K}$, it follows that
$$
  \int_0^T \int_{\pa\Omega}|F(M_m)-F(\widetilde{M}_m)| dx dt 
	\le C (T \m(\partial \Omega) \eta_m)^{1/2} \to 0 \quad \mbox{as }m \to \infty,
$$
which proves the claim.

Now, as $\Omega$ is assumed to be a polygonal domain, $\pa\Omega$ consists of a 
finite number of faces denoted by $(\Gamma_i)_{1 \leq i \leq I}$. 
Similarly to \cite{BCH13,EGHM02}, we define for $\eps>0$ the subset $\Omega_{i,\eps}$ 
of $\Omega$ such that every $x \in \Omega_{i,\eps}$ satisfies 
$\dist(x,\Gamma_i)< \eps$ and $\dist(x,\Gamma_i) < \dist(x,\Gamma_j)$ for all 
$j \neq i$. We also define the subset $\omega_{i,\eps} \subset \Omega_{i,\eps}$ 
as the largest cylinder of width $\eps$ generated by $\Gamma_i$. Let $\nu_i$ be 
the unit vector that is normal to $\Gamma_i$, i.e., more precisely, we introduce 
the set
$$
  \omega_{i,\eps} := \big\{ x - h \nu_i \in \Omega_i : x \in \Gamma_i, \ 
	0 < h < \eps \mbox{ and } [x,x-h \nu_i] \subset \overline{\Omega}_{i,\eps} 
	\big\} \quad \mbox{for all }1 \leq i \leq I. 
$$
Finally, we also introduce the subset $\Gamma_{i,\eps} := \pa\omega_{i,\eps} 
\cap \Gamma_i$, which fulfills $\m(\Gamma_i \setminus \Gamma_{i,\eps}) 
\leq C \eps$ for some constant $C>0$ only depending on $\Omega$.

Let $i \in \{1,\ldots,I\}$ be fixed and let $\psi \in C^\infty_0(\Gamma_i\times(0,T))$.
Then there exists $\eps^*=\eps^*(\psi) > 0$ such that for every 
$\eps \in (0,\eps^*)$, we have $\operatorname{supp}(\psi) \subset\Gamma_{i,\eps}
\times(0,T)$. We write
\begin{align*}
  &\int_0^T \int_{\Gamma_i} (F(\widetilde{M}_m)-F(M)) \psi dxdt 
	= B_{1,m,\eps} + B_{2,m,\eps} + B_{3,\eps}, \quad\mbox{where} \\
  & B_{1,m,\eps} = \int_0^T \frac{1}{\eps} \int_{\Gamma_{i,\eps}} \int_0^\eps 
	\big(F(\widetilde{M}_m(x,t)) - F(M_m(x-h \nu_i,t)) \big) \psi(x,t) dh dx dt, \\
  & B_{2,m,\eps} = \int_0^T \frac{1}{\eps} \int_{\Gamma_{i,\eps}} \int_0^\eps 
	\big( F(M_m(x-h \nu_i,t)) - F(M(x-h \nu_i,t)) \big) \psi(x,t) dh dx dt, \\
  & B_{3,\eps} = \int_0^T \frac{1}{\eps} \int_{\Gamma_{i,\eps}} \int_0^\eps 
	\big( F(M(x-h \nu_i,t)) - F(M)\big) \psi(x,t) dh dx dt.
\end{align*}
We apply the Cauchy--Schwarz inequality to the first term and then use
\cite[Lemma~4.8]{BCH13} and Lemma \ref{lem.unif1} to find that
\begin{align*}
  |B_{1,m,\eps}| &\le \bigg(\int_0^T \frac{1}{\eps} \int_{\Gamma_{i,\eps}} \int_0^\eps 
	\big(F(\widetilde{M}_m(x,t)) - F(M_m(x-h \nu_i,t))\big)^2 dh dx dt\bigg)^{1/2} \\
	&\phantom{xx}{}\times\bigg(\int_0^T \int_{\Gamma_i} \psi(x,t)^2 dxdt\bigg)^{1/2}
  \le \sqrt{\eps + \eta_m}\|F(M_m)\|_{1,2,\M}\|\psi\|_{L^2(\Gamma_i\times(0,T))}.
\end{align*}
Taking into account that Lemma \ref{lem.comp} implies that $F(M_m)\to F(M)$ 
strongly in $L^r(\Omega_T)$ for $1 \le r < 2$, we infer that the second term 
$B_{2,m,\eps}$ converges to zero as $m\to\infty$. This shows that 
$$
  \lim_{m \to \infty} \bigg| \int_0^T \int_{\Gamma_i} (F(\widetilde{M}_m)-F(M)) 
	\psi dxdt \bigg| \le C\sqrt{\eps} + |B_{3,\eps}|.
$$
Since $F(M) \in L^2(0,T;H^1(\Omega))$, the function $F(M)$ has a trace in 
$L^2(\pa\Omega\times(0,T))$ such that $B_{3,\eps} \to 0$ as $\eps \to 0$. Hence,
performing the limit $\eps\to 0$, we conclude that \eqref{6.trace2} holds,
finishing the proof.
\end{proof}

It remains to verify that
the limit function $(S,M)$ obtained in Lemma \ref{lem.comp}
is a weak solution to \eqref{1.eqS}--\eqref{1.f}. We follow the ideas of \cite{CLP03}
and prove that $M$ solves \eqref{weakM}, as the proof of \eqref{weakS} is
analogous. Let $\phi\in C_0^\infty(\Omega\times[0,T))$ and let
$\eta_m=\max\{\Delta x_m,\Delta t_m\}$ be sufficiently small such that
$\operatorname{supp}(\phi)\subset\{x\in\Omega:\dist(x,\pa\Omega)>\eta_m\}\times
(0,T)$. The aim is to prove that $F_{10}^m+F_{20}^m+F_{30}^m\to 0$ as
$m\to\infty$, where
\begin{align*}
  F_{10}^m &= -\int_0^T\int_\Omega M_m\pa_t\phi dxdt
	- \int_\Omega M_m(x,0)\phi(x,0)dx, \\
	F_{20}^m &= d_2\int_0^T\int_\Omega\na^m F(M_m)\cdot\na\phi dxdt, \\
	F_{30}^m &= -\int_0^T\int_\Omega h(S_m,M_m)\phi dxdt.
\end{align*}
The convergence results from Lemma \ref{lem.comp} allow us to perform the limit
$m\to\infty$ in these integrals, leading to
\begin{align*}
  F_{10}^m+F_{20}^m+F_{30}^m &\to -\int_0^T\int_\Omega M\pa_t\phi dxdt
	- \int_\Omega M^0(x)\phi(x,0)dx \\
	&\phantom{xx}{}+ d_2\int_0^T\int_\Omega\na F(M)\cdot\na\phi dxdt
	- \int_0^T\int_\Omega h(S,M)\phi dxdt.
\end{align*}

Now we set $\phi_K^k=\phi(x_K,t_k)$, multiply \eqref{schM} by $\Delta t\phi_K^{k-1}$,
and sum over $K\in\T$ and $k=1,\ldots,N_T$:
\begin{align}
  & F_1^m+F_2^m+F_3^m = 0, \quad\mbox{where} \label{6.weak} \\
	& F_1^m = \sum_{k=1}^{N_T}\sum_{K\in\T} \m(K) (M^k_K-M^{k-1}_K)\phi_K^{k-1},
	\nonumber \\
  & F_2^m = -d_2 \sum_{k=1}^{N_T} \Delta t_m \sum_{K\in\T} 
	\sum_{\sigma \in \E_{{\rm int},K}} \tau_\sigma\DD_{K,\sigma} F(M^k)\phi^{k-1}_K,
	\nonumber \\
  & F_3^m = -\sum_{k=1}^{N_T} \Delta t_m \sum_{K \in \T} \m(K)h(S_K^k,M_K^k)
	\phi_K^{k-1}. \nonumber
\end{align}
We claim that $F_{j0}^m-F_j^m\to 0$ as $m\to\infty$ for $j=1,2,3$. Then
\eqref{6.weak} implies that $F_{10}^m+F_{20}^m+F_{30}^m\to 0$ for $m\to\infty$,
finishing the proof.

For the first limit, we argue as in \cite[Theorem 5.2]{CLP03}:
\begin{align*}
  F_{10}^m &= -\sum_{k=1}^{N_T}\sum_{K\in\T}\m(K)M_{m,K}^k(\phi_K^k-\phi_K^{k-1})
	- \sum_{K\in\T}\m(K)M_{m,K}^0\phi_K^0 \\
	&= -\sum_{k=1}^{N_T}\sum_{K\in\T}\int_{t_{k-1}}^{t_k}\int_K M_{m,K}^k
	\pa_t\phi(x_K,t)dxdt - \sum_{K\in\T}\int_K M_{m,K}^0\phi(x_K,0)dx.
\end{align*}
This shows that $|F_{10}^m-F_1^m|\le C\|\phi\|_{C^2(\overline{\Omega_T})}\eta_m\to 0$
as $m\to\infty$.

Next, we use discrete integration by parts to rewrite
$F_2^m$:
$$
  F_2^m = d_2\sum_{k=1}^{N_T}\Delta t_m\sum_{K\in\T}\sum_{\sigma\in\E_{{\rm int},K}}
	\tau_\sigma\DD_{K,\sigma}F(M^k)\DD_{K,\sigma}\phi^{k-1}.
$$
By the definition of the discrete gradient, we can also rewrite $F_{20}^m$:
$$
  	F_{20}^m = d_2 \sum_{k=1}^{N_T} \sum_{K \in \T}  \sum_{\sigma \in \E_{int,K}} 
		\DD_{K,\sigma} F(M^k)  \frac{\m(\sigma)}{\m(T_{K,\sigma})} 
		\int_{t_{k-1}}^{t_k} \int_{T_{K,\sigma}} \na\phi\cdot\nu_{K,\sigma}dxdt.
$$
Hence, using \cite[Theorem 5.1]{CLP03} and the Cauchy--Schwarz inequality,
we find that
\begin{align*}
  |F_{20}^m-F_2^m| &\le d_2\sum_{k=1}^{N_T} \sum_{K\in\T} 
	\sum_{\sigma\in\E_{{\rm int},K}} \m(\sigma)\DD_\sigma F(M^k) \\ 
	&\phantom{xx}{}\times\bigg|\int_{t_{k-1}}^{t_k}\bigg(\frac{1}{\m(T_{K,\sigma})} 
	\int_{T_{K,\sigma}} \na\phi\cdot \nu_{K,\sigma} dx 
	- \frac{1}{\dist_\sigma}\DD_{K,\sigma} \phi^{k-1}dx \bigg)dt \bigg| \\
	&\le d_2\sum_{k=1}^{N_T} \sum_{K\in\T} \sum_{\sigma\in\E_{{\rm int},K}}  
	\m(\sigma)\DD_\sigma F(M^k)\times C\Delta t_m\eta_m \\
	&\le C\eta_m d_2 \bigg(\sum_{k=1}^{N_T}\Delta t_m \sum_{\sigma\in\E}\m(\sigma) 
	\dist_\sigma \bigg)^{1/2} \bigg( \sum_{k=1}^{N_T}\Delta t_m|F(M^k)|_{1,2,\M}^2 
	\bigg)^{1/2}\\
	&\le C\eta_m d_2\xi^{-1/2} \bigg(\sum_{k=1}^{N_T} \Delta t_m \sum_{\sigma \in \E} 
	\m(\sigma) \dist(x_K,\sigma)\bigg)^{1/2},
\end{align*}
where we used the mesh regularity \eqref{2.regmesh} in the last step.
Taking into account the estimate for $F(M_m)$ from Lemma \ref{lem.grad} and the
property \eqref{2.estmesh}, we infer that $F_{20}^m-F_2^m\to 0$.

Finally, using the regularity of $\phi$, we obtain
\begin{align*}
  |F_{30}^m-F_3^m| &\le \sum_{k=1}^{N_T}\sum_{K\in\T}\m(K)|h(S_K^k,M_K^k)|
	\bigg|\int_{t_{k-1}}^{t_k}\bigg(\phi_K^{k-1}-\frac{1}{\m(K)}\int_K\phi dx\bigg)dt
	\bigg| \\
	&\le \bigg(\kappa_2+\frac{\kappa_3}{\kappa_4}\bigg)\sum_{k=1}^{N_T}\sum_{K\in\T}\m(K)
	\bigg|\int_{t_{k-1}}^{t_k}\bigg(\phi^{k-1}_K-\frac{1}{\m(K)}\int_K\phi dx\bigg)dt
	\bigg| \\
	&\le \bigg(\kappa_2+\frac{\kappa_3}{\kappa_4}\bigg)\m(\Omega)T
	\|\na\phi\|_{L^\infty(\Omega_T)}\eta_m\to 0.
\end{align*}
This finishes the proof.


\section{Numerical experiments}\label{sec.num}

We present in this section some numerical experiments for the biofilm model \eqref{sch.ic}--\eqref{Flux} in one and two space dimensions.

\subsection{Implementation of the scheme}

The finite-volume scheme \eqref{sch.ic}--\eqref{Flux} is implemented in MATLAB. Since the numerical scheme is implicit in time, we have to solve a nonlinear system of equations at each time step. In the one-dimensional case, we use Newton's method. Starting from $(S^{k-1}, M^{k-1})$, we apply a Newton method with precision $\eps = 10^{-10}$ to approximate the solution to the scheme at time step $k$. In the two-dimensional case, we use a Newton method complemented by an adaptive time-stepping strategy to approximate the solution of the scheme at time $t_k$. More precisely, starting again from $(S^{k-1}, M^{k-1})$, we launch a 
Newton method. If the method does not converge with precision 
$\eps= 10^{-8}$ after at most $50$ steps, we multiply the time step by a factor $0.2$ and restart the Newton method. At the beginning of each time step, we increase the value of the previous time step size by multiplying it by $1.1$. Moreover, we impose the condition $10^{-8}\leq \Delta t_{k} \leq 10^{-2}$ with an initial time step size equal to $10^{-5}$. Our adaptive time-step strategy aims to improve the numerical performance of our scheme in terms of number of time steps, CPU time, etc. However, this strategy is not mandatory and, as in our one-dimensional test case, we can always implement our scheme with a constant time step with a reasonable size.

\subsection{Test case 1: Rate of convergence in space}\label{ssec.convergence}

We illustrate the order of convergence in space for the biofilm model in one space dimension with $\Omega = (0,1)$. To this purpose, we choose the coefficients $d_1=4.1667$, $d_2=4.2$, $\kappa_1 = 793.65$, $\kappa_2 = 0.067$, $\kappa_3 = 1$, $\kappa_4 = 0.4$ and $M^D=0$. These values are close to those used in \cite{ESE17}. We take $a=2$ and $b=1$ such that, after elementary computations,
$$
  F(M) = \log(1-x) + \frac{1}{1-x} - 1.
$$
Finally, we impose the initial data
$S^0(x) = 1 - 0.2 \, \sin(\pi x)$ and
\begin{align*}
  & M^0(x) = 0.2 \, g(x-0.38) + 0.9 \, g(x-0.62),\\
  & \mbox{where }g(x)=\max\{1 - 9^2 x^2, 0\}. 
\end{align*}
Since exact solutions to the biofilm model are not explicitly known, we compute a reference solution $(S_{\rm ref}, M_{\rm ref})$ on a uniform mesh composed of $20,480$ cells and with $\Delta t = (1/20,480)^2$. We use this rather small value of $\Delta t$ because the Euler discretization in time exhibits a first-order convergence rate, while we expect a second-order convergence rate in space for scheme \eqref{sch.ic}-\eqref{Flux}, due to to two-point flux approximation scheme used in this work. We compute approximate solutions on uniform meshes made of $80$, $160$, $320$, $640$, $1280$ and $2560$ cells, respectively. In Figure \ref{Fig.conv}, we present the $L^1(\Omega)$ norm of the difference between the approximate solutions and the average of the reference solution $(S_{\rm ref}, M_{\rm ref})$ at the final time $T=10^{-3}$. As expected, we observe a second-order convergence rate in space.

\begin{figure}
\begin{center}
\includegraphics[scale=1]{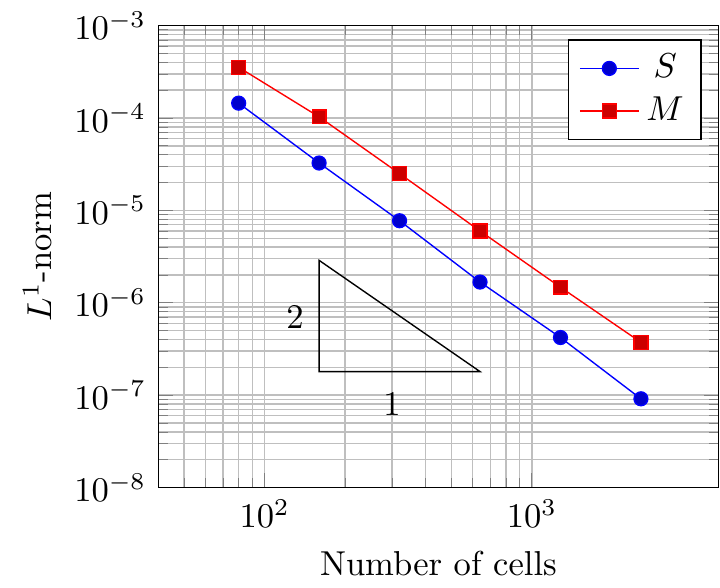}
\end{center}
\caption{Test case 1: $L^1$ norm of the error between the reference solution and the solutions computed on coarser grids at final time $T=10^{-3}$.}
\label{Fig.conv}
\end{figure}

\subsection{Test case 2: Microbial floc}
\label{ssec.2dsimu}

We investigate the behavior of $S$ and $M$ in two space dimensions with domain $\Omega = (0,1) \times (0,1)$ and final time $T=2$. As in the first test case, we choose the coefficients $d_1=4.1667$, $d_2=4.2$, $\kappa_1 = 793.65$, $\kappa_2 = 0.067$, $\kappa_3 = 1$, $\kappa_4 = 0.4$ and $M^D=0$. Here, we take $a=b=4$ such that
$$
  F(M) = -\frac{18 x^2 - 30 x + 13}{3 (x-1)^3} + x + 4 \log(1-x) - \frac{13}{3}
$$
and the initial data $S^0(x,y) = 1$ and 
\begin{align*}
  & M^0(x,y) = 0.3 \, p(x-0.4,y-0.5) + 0.9 \, p(x-0.6,y-0.5),\\
	& \mbox{where }p(x,y)=\max\{1 - 8^2 x^2 - 8^2 y^2, 0\}. 
\end{align*}
The initial data models a microbial floc, i.e.\ a biofilm without substratum.
This situation plays an important role in wastewater treatment.

In Figure \ref{Fig.plotMScase2}, we illustrate the behavior of $S$ and $M$ along time for a mesh of $\Omega = (0,1)^2$ composed of 3584 triangles. We observe, as in \cite{EPL01,EZE09,ESE17}, that after a transient time, the two colonies merge. After this stage, we observe an expansion of the region $\{M > 0 \}$ due to the porous-medium type degeneracy for the equation of $M$, which implies a finite speed of propagation of the interface between $\{M > 0 \}$ and $\{M=0 \}$.
With the chosen parameters, the production rate of the biofilm is positive 
if and only if $S>\kappa^*:=\kappa_2\kappa_4/(\kappa_3-\kappa_2)\approx 0.029$, 
and the biomass fraction is increasing in $\{S>\kappa^*\}$, which is 
confirmed by the numerical experiments.

\begin{figure}
\begin{center}
\includegraphics[scale=0.395]{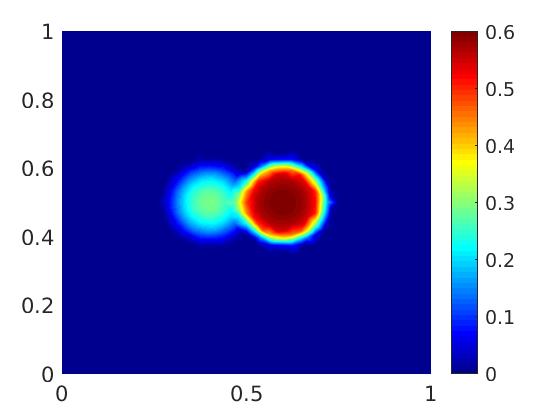}
\includegraphics[scale=0.35]{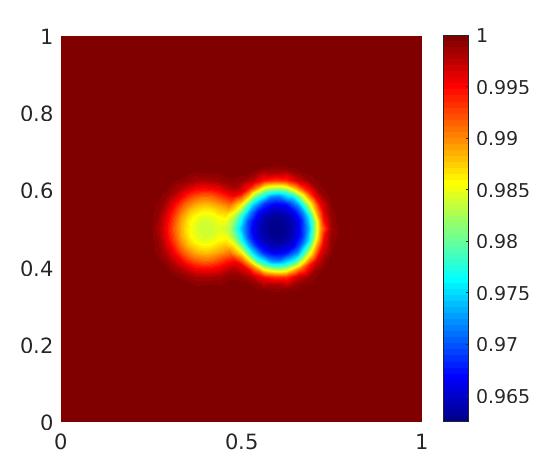}
\includegraphics[scale=0.395]{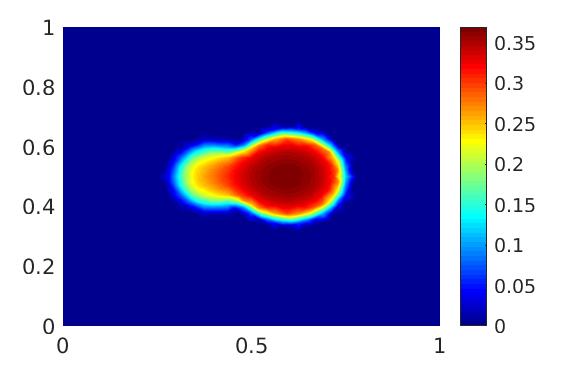}
\includegraphics[scale=0.35]{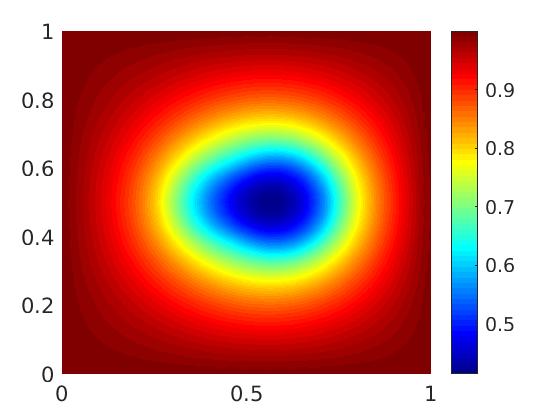}
\includegraphics[scale=0.395]{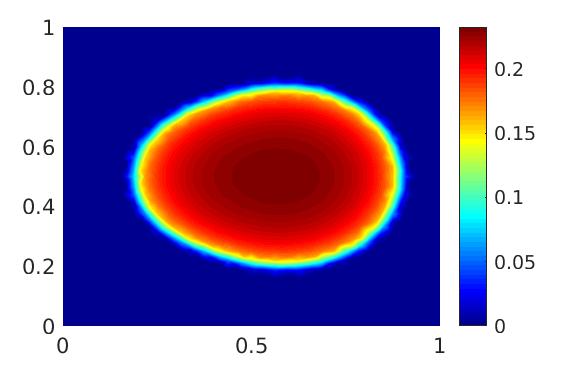}
\includegraphics[scale=0.35]{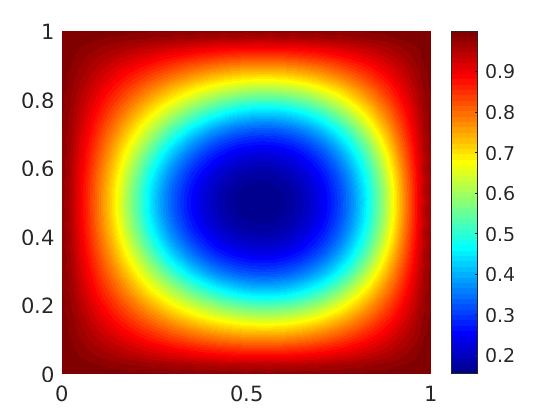}
\end{center}
\caption{Test case 2: Evolution of $M$ (left column) and $S$ (right column) for $t=10^{-4}$ (top row), $t=10^{-2}$ (middle row) and $t=2$ (bottom row).}
\label{Fig.plotMScase2}
\end{figure}



\begin{thebibliography}{11}
\bibitem{AlPe20}
A.~Alhammali and M.~Peszynska. Numerical analysis of a parabolic variational
inequality system modeling biofilm growth at the porescale.
{\em Numer. Meth. Partial Differ. Eqs.} 36 (2020), 941--971.

\bibitem{AES18}
M.~Ali, H.~Eberl, and R.~Sudarsan. Numerical solution of a degenerate, diffusion
reaction based biofilm growth model on structured non-orthogonal grids.
{\em Commun. Comput. Phys.} 24 (2018), 695--741.

\bibitem{ACM17} B.~Andreianov, C.~Canc\`es, and A.~Moussa. A nonlinear time
compactness result and applications to discretization of degenerate 
parabolic-elliptic PDEs. {\em J. Funct. Anal.} 273 (2017), 3633--3670.

\bibitem{BCF15} M.~Bessemoulin-Chatard, C.~Chainais-Hillairet, and F.~Filbet.
On discrete functional inequalities for some finite volume schemes.
{\em IMA J. Numer. Anal.} 35 (2015), 1125--1149.

\bibitem{BCH13} K.~Brenner, C.~Canc\`es, and D.~Hilhorst. Finite volume approximation 
for an immiscible two-phase flow in porous media with discontinuous capillary 
pressure. {\em Comput. Geosci.} 17 (2013), 573--597.


\bibitem{CLP03} C.~Chainais-Hillairet, J.-G.~Liu, Y.-J.~Peng. 
Finite volume scheme for multi-dimensional drift-diffusion equations and convergence 
analysis. {\em ESAIM Math. Model. Numer. Anal.} 37 (2003), 319--338.

\bibitem{CRNR13} F.~Clarelli, C.~Di Russo, R.~Natalini, and M.~Ribot. 
A fluid dynamics model of the growth of phototrophic biofilms.
{\em J. Math. Biol.} 66 (2013), 1387--1408.

\bibitem{DJZ21}
E.~S.~Daus, A.~J\"ungel, and A.~Zurek. Convergence of a finite-volume scheme for a 
degenerate-singular cross-diffusion system for biofilms. 
{\em IMA J. Numer. Anal.} 41 (2021), 935--973.

\bibitem{DMZ19} E.~S.~Daus, J.-P.~Mili\v{s}i\'c, and N.~Zamponi. Analysis of 
a degenerate and singular volume-filling cross-diffusion system modeling 
biofilm growth. {\em SIAM J. Math. Anal.} 51 (2019), 3569--3605.

\bibitem{Dei85} K.~Deimling. {\em Nonlinear Functional Analysis}. Springer, 
Berlin, 1985.

\bibitem{DuEb06} A.~Duvnjak and H.~Eberl. Time-discretization of a degenerate
reaction-diffusion equation arising in biofilm modeling. 
{\em Electron. Trans. Numer. Anal.} 23 (2006), 15--37.

\bibitem{EbDe07} H.~Eberl and L.~Demaret. A finite difference scheme for a
degenerated diffusion equation arising in microbial ecology.
{\em Electron. J. Differ. Eqs.} 15 (2007), 77--95.

\bibitem{EEWZ14}
H.~Eberl, M.~Efendiev, D.~Wrzosek, and A.~Zhigun. Analysis of a degenerate
biofilm model with a nutrient taxis term. 
{\em Discrete Cont. Dyn. Sys.} 34 (2014), 99--119.

\bibitem{EPL01}
H.~Eberl, D.~Parker, and M.~van Loosdrecht. A new deterministic spatio-temporal 
continuum model for biofilm development. {\em Comput. Math. Meth. Medicine} 3
(2001), no.~429794, 15 pages. 

\bibitem{EZE09}
M.~Efendiev, S.~Zelik, and H.~Eberl. Existence and longtime behavior of a biofilm
model. {\em Commun. Pure Appl. Anal.} 8 (2009), 509--531.

\bibitem{ESE17}
B.~Emerenini, S.~Sonner, and H.~Eberl. Mathematical analysis of a quorum sensing induced biofilm dispersal model and numerical simulation of hollowing effects. {\em Math. Biosci. Eng.} 14 (2017), 625--653.

\bibitem{EFGHG13} R.~Eymard, P.~F\'eron, T.~Gallou{\"e}t, R.~Herbin, and C.~Guichard.
Gradient schemes for the Stefan problem. {\em Int. J. Finite Vol.} 10 (2013), 37 pages.

\bibitem{EGH00} R.~Eymard, T.~Gallou{\"e}t, and R.~Herbin. Finite volume methods. 
In: P.~G.~Ciarlet and J.-L.~Lions (eds.),
{\em Handbook of Numerical Analysis} 7 (2000), 713--1018.

\bibitem{EGHM02} R.~Eymard, T.~Gallou{\"e}t, R.~Herbin, and A.~Michel. Convergence of finite volume schemes for parabolic
degenerate equations. {\em Numer. Math.}, 92 (2002), 41--82.



\bibitem{GSE18}
M.~Ghasemi, S.~Sonner, and H.~Eberl. Time adaptive numerical solution of a highly
non-linear degenerate cross-diffusion system arising in multi-species biofilm modelling.
{\em Eur. J. Appl. Math.} 29 (2018), 1035--1061.

\bibitem{GKN18}
M.~Gokieli, N.~Kenmochi, and M.~Niezg\'odka. Mathematical modeling of biofilm 
development. {\em Nonlin. Anal. Real World Appl.} 42 (2018), 422--447.


\bibitem{Jue16} A.~J\"ungel. {\em Entropy Methods for Diffusive Partial Differential
Equations}. BCAM Springer Briefs, Springer, 2016.

\bibitem{RaEb14} K.~Rahman and H.~Eberl. Numerical treatment of a cross-diffusion model 
of biofilm exposure to antimicrobials. In: R.~ Wyrzykowski, J.~Dongarra, K.~Karczewski, 
and J.~Wa\'sniewski (eds.), {\em Parallel Processing and Applied Mathematics}. 
Lect. Notes Comput. Sci. 8384, Part I, pp.~134--144. Springer, Heidelberg, 2014. 

\bibitem{RSE15} K.~Rahman, R.~Sudarsan, and H.~Eberl. A mixed-culture biofilm 
model with cross-diffusion. {\em Bull. Math. Biol.} 77 (2015), 2086--2124.

\bibitem{Sch19}
R.~Schulz. Biofilm modeling in evolving porous media with Beavers--Joseph
condition. {\em Z. Angew. Math. Mech.} 99 (2019), no.~e20180123, 19 pages.

\bibitem{WEMNPRL06} O.~Wanner, H.~Eberl, E.~Morgenroth, D.~Noguera, C.~Picioreanu, 
B.~Rittmann, and M.~van Loosdrecht. {\em Mathematical Modeling of Biofilms}. 
IWA Publishing, London, 2006.

\bibitem{WaGu86} O.~Wanner and W.~Gujer. A multispecies biofilm model. 
{\em Biotechnol. Bioengin.} 28 (1986), 314--328.

\bibitem{Zha12} T~Zhang. Modeling of biocide action against biofilm. 
{\em Bull. Math. Biol.} 74 (2012), 1427--1447.
\end{thebibliography}
\end{document}